\documentclass[11pt]{article}
\usepackage{amsmath,amsthm,amssymb}
\usepackage[english]{babel}
\usepackage{color}
\usepackage{graphicx}
\usepackage[top=3cm, bottom=3cm, left=2.3cm, right=2.3cm]{geometry}
\numberwithin{equation}{section}
\newtheorem{theorem}{Theorem}[section]
\newtheorem{lemma}[theorem]{Lemma}
\newtheorem{corollary}[theorem]{Corollary}

\newtheorem{definition}[theorem]{Definition}
\newtheorem{remark}[theorem]{Remark}

\newtheorem{assumption}[theorem]{Assumption}

\def\bfR{{\mathbb R}}
\def\pp{\mbox{\bf P}}
\def\ee{\mbox{\bf E}}
\def\dd{\mbox{\rm d}}

\newcommand*{\fatdot}{ \raisebox{-1pt}{\makebox[.5ex]{\textbf{$\cdot$}}} }
\begin{document}
\title{Extension Technique for Functions of Diffusion Operators:\\ a stochastic approach}
\author{Sigurd Assing \& 
John Herman\footnote{Supported by EPSRC funding as a part of the MASDOC DTC, Grant reference number EP/HO23364/1.}\\
Department of Statistics, The University of Warwick,
Coventry CV4 7AL, UK\\
e-mail: s.assing@warwick.ac.uk}
\date{}
\maketitle

\vspace{-1cm}
\begin{abstract}
It has recently been shown that complete Bernstein functions
of the Laplace operator map the  Dirichlet boundary condition
of a related elliptic PDE to the Neumann boundary condition.
The importance of this mapping consists in being able to convert problems 
involving non-local operators, like fractional Laplacians, 
into ones that only involve differential operators.
We generalise this result to diffusion operators
associated with stochastic differential equations,
using a method which is entirely based on stochastic analysis.
\end{abstract}
\noindent
{\large KEY WORDS}\hspace{0.2cm}
Dirichlet-to-Neumann map,
elliptic equation,
trace process,
Krein strings

\vspace{10pt}
\section{Introduction}
This paper is strongly motivated by Kwa\'snicki \&\ Mucha's recent work, \cite{KM2018}, 
where they look for functions $\psi$ such that $\psi(-\Delta)$
can be both well-defined and associated with a local PDE operator
via the so-called Dirichlet-to-Neumann map, by extension method.
Here, $\Delta$ stands for the Laplace operator on $\bfR^d$, 
and they found that the above works for 
complete Bernstein functions.

Our contribution to this problem is two-fold.

First, we show that, for the majority of complete Bernstein functions $\psi$,
the above Dirichlet-to-Neumann map also works
when $\Delta$ is replaced by a diffusion operator of type
\[
\sum_{i=1}^d a_i(x)\,\frac{\partial}{\partial x_i}
\,+\,
\frac{1}{2}\sum_{i,j=1}^d(\sigma\sigma^{\bf T})_{ij}(x)\,\frac{\partial^2}{\partial x_i \partial x_j}\,,
\]
where the coefficients are continuous functions on $\bfR^d$
satisfying a linear growth condition.

Second, our method of proof is based on a stochastic approach totally different to
the purely analytic approach taken in \cite{KM2018}. Though the basic ideas
for our stochastic approach have been around for a while in different contexts---see
below for more details---it is technically quite demanding to push the method
through for general complete Bernstein functions. In an abstract way,
the mechanism connecting the stochastic and analytic approach used by us
and the authors of \cite{KM2018}, respectively, 
is the one-to-one relationship 
of local times of generalised diffusions and complete Bernstein functions;
and this relationship was first verified in a rather hidden application
given in \cite{KW1982}.

In the remaining part of the introduction we put the Dirichlet-to-Neumann map
into context, in Section 2 we formulate our results, in Section 3 we prove them,
in Section 4 we discuss our conditions, and we finish the paper with some
auxiliary results given in the Appendix.

Consider a bounded continuous  function $u$ on $\bfR^d\times[0,\infty)$
which is harmonic in $\bfR^d\times(0,\infty)$, i.e.
\begin{equation}\label{firstPDE}
\Delta_x\,u(x,y)+\partial^2_y\,u(x,y)\,=\,0,
\quad(x,y)\in\bfR^d\times(0,\infty),
\end{equation}
where $\Delta_x$  denotes the Laplace operator acting on the $d$-dimensional
$x$-component of $u$. Then it has been
long\footnote{We could not track the first occasion of this classic result.}
known that the square root of the Laplacian, $-(-\Delta_x)^{1/2}$,
can map the Dirichlet boundary condition pointwise to the
Neumann boundary condition of the elliptic PDE (\ref{firstPDE}), i.e.
\[
-(-\Delta_x)^{1/2}\,u(\cdot,0)\,=\,\lim_{y\downarrow 0}\partial_y\,u(\cdot,y),
\]
and this map is the earliest example of the so-called 
Dirichlet-to-Neumann map.

This example of course triggered the question  
whether there are functions $\psi$, other than the square root,
such that $-\psi(-\Delta_x)$ maps the Dirichlet boundary condition to the 
Neumann boundary condition of an associated PDE,
and the first systematic answer to this question was given by
Caffarelli \&\ Silvestre in their 2007 paper \cite{CS2007}:
they were able to construct a mapping for any fractional Laplacian,
$-(-\Delta_x)^{\alpha/2},\,0<\alpha<2$,
their associated PDE read
\[
\Delta_x\,u
+\frac{1-\alpha}{y}\partial_y\,u
+\partial^2_y\,u\,=\,0,
\quad\mbox{in}\quad\bfR^d\times(0,\infty),
\]
and they proved 
\[
-(-\Delta_x)^{\alpha/2}\,u(\cdot,0)
\,=\,
c_\alpha\lim_{y\downarrow 0}\,[y^{1-\alpha}\partial_y\,u(\cdot,y)].
\]

So, they did not map the Dirichlet boundary condition to a proper 
Neumann boundary condition, they rather mapped it to a weighted
Neumann boundary condition.
However, the importance of their mapping consists in being able to convert
problems involving non-local operators, like fractional Laplacians,
into ones that only involve (maybe degenerate) differential operators.

Since then, in particular because of this link between non-local and local operators,
there has been huge activity generalising the above extension method in many ways, 
mostly replacing $\Delta_x$ by more general diffusion operators, 
but also considering equations in more general function spaces. 
Avoiding to quote every of these many activities, 
the made progress can be seen best when looking into applications, 
and we concentrate on the two areas of application we would value most:
analysing fractional obstacle problems (and related optimal stopping problems),
see \cite{S2007,CSS2008,CF2013,CRS2017,BFR2018};
proving Harnack inequalities and maximum principles with respect to fractional operators,
see \cite{CS2007,ST2010,SZ2013,FF2015,BMST2016,MS2017}.

However, in all this work,  the authors restricted themselves to fractional powers,
and the only work we are aware of, where functions $\psi$ other than fractional powers
were systematically treated, is Kwa\'snicki \&\ Mucha's recent paper \cite{KM2018}.
There is also a follow-up paper, \cite{K2019}, on ArXiv, by Kwa\'snicki,
and we will compare the results in both papers with ours in Section 4.
Below, we only add a few words on the origins of our stochastic approach.

The first paper to be mentioned is \cite{MO1969}, by Molchanov \&\ Ostrovskii,
in 1969. From the analytic point of view, this paper deals with
fractional Laplacians like the 2007-paper \cite{CS2007}, by Caffarelli \&\ Silvestre,
and they applied a technique which in some way is the stochastic analogue 
to the extension method.

While the heuristic idea of the extension method is to study an operator
acting on functions on $\bfR^d$ via a PDE in $\bfR^d\times(0,\infty)$,
the stochastic analogue of this idea would be to study the Markov process 
associated with that operator via a pair of stochastic processes
taking values in $\bfR^d\times(0,\infty)$. The process associated with 
the operator of interest would be the trace of the corresponding
pair of processes on the boundary $\bfR^d\times\{0\}$ of
$\bfR^d\times(0,\infty)$---see Section 2 for precise definitions 
within our context.

However, Molchanov \&\ Ostrovskii did not make a connection
between generators associated with trace processes 
and Neumann boundary conditions of solutions to PDEs.
The only paper we are  aware of, where this connection has been made
in a stochastic setting, is the 1986-paper \cite{H1986}, by Hsu.
Formally, the connection is easily made
by combining It\^o's formula and random time change.
The extra difficulty in Hsu's case was that,
instead of $\bfR^d\times(0,\infty)$,
he considered bounded domains in $\bfR^{d+1}$ 
with sufficiently smooth boundary.
In our case, because of the general setup,
the difficulty comes from a lack of regularity,
which has made it harder to apply It\^o's formula.

{\bf Acknowledgement}
We are grateful to the anonymous referee for their detailed report
which helped to clarify the presentation. Moreover, the fact, quoted in  Remark \ref{DtoNtrue}(a),
that uniform convergence can even follow from pointwise convergence,
under certain conditions,
was kindly pointed out to us by the referee.
\section{Results}\label{rr}
Our stochastic extension method is based on four ingredients.

First, let $X=[X(t),\,t\ge 0]$ be a diffusion, taking values in $\bfR^d$, 
which is a global solution of the SDE
\begin{equation}\label{SDE}
\dd X_i(t)
\,=\,
\sum_{k=1}^p\sigma_{ik}(X(t))\,\dd B_k(t)
+a_{i}(X(t))\,\dd t,
\quad
i=1,\dots,d,
\end{equation}
where $(B_1,\dots,B_p)^{\bf T}$ is a Brownian motion, 
$\sigma=(\sigma_{ik})$ is a $d\times p\;-\,$matrix of functions on $\bfR^d$,
and $a=(a_1,\dots,a_d)^{\bf T}$ is a vector field on $\bfR^d$. Of course, 
the infinitesimal operator associated with this diffusion formally reads
\begin{equation}\label{formal generator}
{\cal L}_x\,=\,a\cdot\nabla_{\!x}
\,+\,
\frac{1}{2}\sum_{i,j=1}^d(\sigma\sigma^{\bf T})_{ij}(x)\,\frac{\partial^2}{\partial x_i \partial x_j}\,.
\end{equation}
Assume that both coefficients $a,\sigma$ are continuous 
but also satisfy some growth condition to allow for global solutions of (\ref{SDE}).

Second, let $\hat{m}$ be a Krein string,
that is a monotone non-decreasing right-continuous function
$\hat{m}:[0,\infty]\to[0,\infty]$ such that $\hat{m}(\infty)=\infty$.
Either $\hat{m}\equiv\infty$,
or $\hat{m}$ is uniquely associated with a locally finite measure
on $([0,R),{\cal B}([0,R)))$, where
$R=\sup\{y\ge 0:\hat{m}(y)<\infty\}>0$.
If $R<\infty$, set $y_1=R$,
and otherwise, if the measure is not identically zero,
let $y_1>0$ be the right endpoint of its support in $[0,\infty)$.
This way, except for the trivial strings
$\hat{m}\equiv\infty$, and
$\hat{m}(y)=0,\,\forall y\in[0,\infty)$, 
any string is also uniquely associated with a measure on 
$([0,y_1]\cap\bfR,{\cal B}([0,y_1]\cap\bfR))$,
whose restriction to $[0,y_1)$ is locally finite.
In what follows, excluding the two trivial strings,
we are going to use the same notation $\hat{m}$ 
for both string and measure because it will always be clear 
from the context which one of the two objects is meant.
Unless stated otherwise, the measure-version of $\hat{m}$ refers to the measure on
$([0,y_1]\cap\bfR,{\cal B}([0,y_1]\cap\bfR))$.

Third, consider the elliptic PDE 
\begin{equation}\label{PDE}
{\cal L}_x u\times\hat{m}(\dd y)+\partial_y^2 u\,=\,0,
\quad\mbox{in}\quad
\bfR^d\times(0,y_1),
\end{equation}
where the product ${\cal L}_x u(x,y)\times\hat{m}(\dd y)$ is understood
in the sense of distributions with respect to $y$, for any fixed $x$.
We are looking for solutions of the following type:
\begin{definition}\rm\label{solution}
A function $u:\bfR^d\times(0,y_1)\to\bfR$ such that
\begin{itemize}
\item
$u(\cdot,y)\in C^2(\bfR^d)$, for any $y\in(0,y_1)$,
\item
$u(x,\cdot),\,\partial_i u(x,\cdot),\,\partial_{ij}u(x,\cdot),\,1\le i,j\le d$,
are c\`adl\`ag functions\footnote{In what follows, 
$\partial_i$ and $\partial_{ij}$ is short notation for 
$\partial/\partial x_i$ and $\partial^2/\partial x_i\partial x_j$, respectively.}
on $(0,y_1)$, for any $x\in\bfR^d$,
\end{itemize}
is said to be a {\it solution of/to (\ref{PDE})} if and only if
\[
\int_{(0,y_1)}{\cal L}_x u(x,y)\,g(y)\,\hat{m}(\dd y)
+
\int_0^{y_1}u(x,y)\,g''(y)\,\dd y
\,=\,
0,
\]
for any $x\in\bfR^d$, and any smooth function $g:(0,y_1)\to\bfR$ with compact support.
\end{definition}

Of course, if $u$ solves (\ref{PDE}) in the sense of Definition \ref{solution}, 
then $\partial_y^2\,u(x,\cdot)$ is a locally finite signed measure on $((0,y_1),{\cal B}((0,y_1)))$,
for any $x\in\bfR^d$. Therefore, if $x\in\bfR^d$ is fixed, the partial derivatives 
$\partial_y^+ u(x,y)$ (from right) and $\partial_y^- u(x,y)$ (from left) exist, 
for all $y\in(0,y_1)$, and
\begin{equation}\label{L dm integral}
\partial_y^2 u(x,(y_\star,y^\star])
\,=\,
\partial_y^+ u(x,y^\star)-\partial_y^+ u(x,y_\star)
\,=\,
-\int_{(y_\star,y^\star]}{\cal L}_x u(x,y)\,\hat{m}(\dd y),
\end{equation}
for all $y_\star,y^\star$ such that $0<y_\star<y^\star<y_1$.
\begin{remark}\rm\label{implicit conditions}
Existence of solutions to (\ref{PDE}) in the sense of Definition \ref{solution}
implicitly requires the coefficients of the operator ${\cal L}_x$ 
to be `good enough'. In this paper, the only explicit assumption 
on these coefficients is continuity. 
All other assumptions are made implicitly via properties of solutions to equations.
First, we require existence of global solutions to the SDE (\ref{SDE}),
but all other implicit assumptions will be made via properties
of solutions to the PDE (\ref{PDE}).
\end{remark}

Fourth, let $Y=[Y(t),\,t\ge 0]$ be another diffusion, which is independent of $X$,
and which takes values in the interval $[0,y_1]\cap\bfR$,
where $y_1>0$ can be finite or infinite.
Suppose that this diffusion is regular in $(0,y_1)$, 
right-singular but not absorbing at zero,
and left-singular at $y_1$.

Such a one\,-\,dimensional diffusion can be uniquely determined by 
a scale function and a speed measure---see \cite{AS1998} for a general account
on this theory based on stochastic calculus. For the purpose of this paper,
we choose the scale function to be the identity, 
and we want to construct $Y$ as in the proof of Lemma A.1 in the Appendix,
using a measure $\hat{m}$ on $([0,y_1]\cap\bfR,{\cal B}([0,y_1]\cap\bfR))$
which is associated with a non-trivial Krein string.

This construction is a special case of \cite[Theorem (6.5)]{AS1998},
and following the terminology given in \cite{AS1998},
this theorem tells us in particular
that the corresponding speed measure would not be $\hat{m}$ but
\[
m(\dd y)\,=\left\{\begin{array}{lcc}
\hat{m}(\{0\})\,\delta_0(\dd y)+\frac{1}{2}\hat{m}(\dd y){\bf 1}_{(0,y_1)}
&:&
\mbox{$y_1=\infty$, or $y_1<\infty$ absorbing,}\\
\rule{0pt}{12pt}
\hat{m}(\{0\})\,\delta_0(\dd y)+\frac{1}{2}\hat{m}(\dd y){\bf 1}_{(0,y_1)}
+\hat{m}(\{y_1\})\,\delta_{y_1}(\dd y)
&:&
\mbox{else},
\end{array}\right.
\]
where $m$ is defined on ${\cal B}([0,y_1))$ in the first case,
and on ${\cal B}([0,y_1])$ in the second.
But, according to \cite[Proposition (5.34)]{AS1998}, 
for our choice of regular and singular points,
this speed measure would have to satisfy:
\begin{itemize}\label{speedDefi}
\item[(S1)]
$m(K)<+\infty$, for any compact subset $K\subseteq[0,y_1)$;
\item[(S2)]
$m([0,y_1])<+\infty$, if $y_1<\infty$ is not absorbing;
\item[(S3)]
$m(O)>0$, for any open subset $O\subseteq(0,y_1)$.
\end{itemize}
So, we should ask whether the wanted regularity of $Y$
in terms of regular and singular points puts
a constraint on the type of strings for which our method works.

That said, by definition of $y_1$,
any non-trivial Krein string $\hat{m}$ is locally finite on $[0,y_1)$,
so (S1) is always satisfied, at least. 
Property (S2) is not restrictive either: 
strings which do not lead to this property of $m$ would be related to
the case where $y_1$ is absorbing for $Y$.
Property (S3) though, which is due to the wanted regularity of $Y$ in $(0,y_1)$,
requires the string to be strictly increasing on $(0,y_1)$.

In this paper, we therefore restrict ourselves to 
non-trivial Krein strings which are strictly increasing on $(0,y_1)$.
Note that, by \cite[Theorem (6.5)]{AS1998}, using such strings
in the construction described in the proof of Lemma A.1 in the Appendix
indeed gives diffusions $Y$ of the wanted regularity.

However, another constraint becomes necessary when 
the diffusion $Y$ is supposed to satisfy
a stochastic differential equation which is desirable for our technique.

Observe that any measure $m$ satisfying (S1-3) admits  
a Lebesgue decomposition into the sum of two singular measures,
one of which is absolutely continuous with respect to the Lebesgue measure. 
It is a remarkable consequence of \cite[Theorem (7.9)]{AS1998}
that $Y$ would only satisfy an appropriate stochastic differential equation if
\begin{itemize}
\item[(S4)]
the density of the absolutely continuous part
is Lebesgue-almost everywhere positive in $(0,y_1)$.
\end{itemize}

Assuming (S4), without restricting generality,
using a measurable function $b:[0,y_1]\cap\bfR\to\bfR$,
we can write the Lebesgue decompositions of both measures $m$ and $\hat{m}$
in the form
\[
m(\dd y)\,=\left\{\begin{array}{lcc}
\frac{1}{2}b^{-2}(y)\dd y+m_0\,\delta_0(\dd y)+\frac{1}{2}n(\dd y)
&:&
\mbox{$y_1=\infty$, or $y_1<\infty$ absorbing,}\\
\rule{0pt}{12pt}
\frac{1}{2}b^{-2}(y)\dd y+m_0\,\delta_0(\dd y)+\frac{1}{2}n(\dd y)+m_1\,\delta_{y_1}(\dd y)
&:&
\mbox{else},
\end{array}\right.
\]
and
\[
\hat{m}(\dd y)\,=\left\{\begin{array}{lcl}
b^{-2}(y)\dd y+m_0\,\delta_0(\dd y)+n(\dd y)
&:&
y_1=\infty,\\
\rule{0pt}{12pt}
b^{-2}(y)\dd y+m_0\,\delta_0(\dd y)+n(\dd y)+m_1\,\delta_{y_1}(\dd y)
&:&
y_1<\infty,
\end{array}\right.
\]
respectively,
where $m_0=\hat{m}(\{0\})=\hat{m}(0),\,m_1=\hat{m}(\{y_1\})=\hat{m}(y_1)-\hat{m}(y_1-)$,
and the measure $n$ satisfies $n([0,y_1]\cap\bfR\setminus{\cal N})=0$, 
for some Borel set ${\cal N}\subseteq(0,y_1)$ of Lebesgue measure zero.
Note that, since $b$ maps into $\bfR$, the density $b^{-2}(y)$ is positive, 
for all $y\in[0,y_1]\cap\bfR$.
Furthermore, for technical reasons, we set $b(y)=0$, 
for all $y\in({\cal N}\cup\{0,y_1\})\cap\bfR$,
where $b^{-2}(y)=+\infty$, if $b(y)=0$, as usual.

Eventually, if the Lebesgue decomposition of the speed measure $m$ 
is written in the above form, then \cite[Theorem (7.9)]{AS1998} yields 
that $Y$ solves an SDE of type
\begin{equation}\label{SDE-Y}
\dd Y(t)\,=\,\sqrt{2}\,b(Y(t))\,\dd B(t)+\dd L^0_t(Y)-\dd L^{y_1}_t(Y),
\end{equation}
where $B$ is a Brownian motion, 
and $[L^y_t(Y),\,t\ge 0]$ stands for the symmetric local time of $Y$ at $y\in[0,y_1]$.
Further details on how this local time is defined
can be found in the proof of Lemma A.1 in the Appendix. 
\begin{remark}\rm
(a)
Though we write the Lebesgue decomposition of the speed measure in a slightly different way,
\cite[Theorem (7.9)]{AS1998} can still be applied in our context.
Note that $L^{y_1}_t(Y)=0,\,t\ge 0$, if $y_1=\infty$,
and that one can assume that $(B_1,\dots,B_p)$ and $B$ are independent,
without loss of generality.

(b)
Of course, the law of $Y$ is uniquely determined by the chosen Krein string
via the corresponding speed measure. But equation (\ref{SDE-Y}) will always
admit solutions with other laws, too, because the coefficient $b$
has got at least one zero.
\end{remark}

Our first goal is to establish a version of It\^o's lemma for $u(X(\cdot),Y(\cdot))$,
when $u$ is a solution to (\ref{PDE}). Solutions of (\ref{PDE}) are jointly
measurable, they are continuous in $x$, and also continuous in $y$ 
(recall that $\partial_y^\pm u$ do exist),
but they might \underline{not} be jointly continuous.
This suggests that more regularity than stated in Definition \ref{solution} 
would be needed for It\^o's lemma to hold true.
We should nonetheless try to keep assumptions on the regularity of $u$
as `weak' as possible. Adding the following condition seems to be enough.
\begin{assumption}\rm\label{BCxreg}
The functions
$u,\,\partial_i u,\,\partial_{ij}u,\,1\le i,j\le d$,
can be extended to locally bounded functions on 
$\bfR^d\times\left(\rule{0pt}{11pt}[0,y_1]\cap\bfR\right)$ 
by taking the limits
\[
\mbox{$\lim_{y\downarrow 0}$}\,u(x,y),\quad
\mbox{$\lim_{y\downarrow 0}$}\,\partial_i u(x,y),\quad
\mbox{$\lim_{y\downarrow 0}$}\,\partial_{ij}u(x,y),
\]
and
\[
\mbox{$\lim_{y\uparrow y_1}$}\,u(x,y),\quad
\mbox{$\lim_{y\uparrow y_1}$}\,\partial_i u(x,y),\quad
\mbox{$\lim_{y\uparrow y_1}$}\,\partial_{ij}u(x,y),
\]
at any fixed $x\in\bfR^d$,
where the latter three limits are only taken when $y_1<\infty$.
\end{assumption}
\begin{remark}\rm\label{afterBCxreg}
Under the above assumption, the particular limit of $u(x,y),\,y\downarrow 0$, exists, 
which we will also denote by $f(x)$. So $f:\bfR^d\to\bfR$ is a locally bounded
measurable function which plays the role of a Dirichlet boundary condition for the solution $u$.
In what follows, to emphasise this role, we will always write $u_f$ for the extension of $u$,
when Assumption \ref{BCxreg} is assumed. Using $u_f(x,0)$ as an alternative notation  
for the limit of $u_f(x,y),\,y\downarrow 0$, leads to the wanted equality 
$u_f(x,0)\,=\,f(x),\,x\in\bfR^d$. The other limits taken under Assumption \ref{BCxreg}
are going to be denoted by
$\partial_i u_f(x,0),\,\partial_{ij}u_f(x,0)$ and, when $y_1<\infty$, by
$u_f(x,y_1-),\,\partial_i u_f(x,y_1-),\,\partial_{ij}u_f(x,y_1-)$,
for all $x\in\bfR^d,\,1\le i,j\le d$.
The notations $\mathcal{L}_x u_f(\cdot, 0)$ and ${\cal L}_x u_f(\cdot,y_1-)$
used further below would refer to these limits, too.
\end{remark}
\begin{corollary}\label{BCyreg}
Let $u_f$ be the extension of
a solution $u$ to (\ref{PDE}) satisfying Assumption \ref{BCxreg}.
Then, the limit
$\partial_y^+ u_f(x,0)=\lim_{y\downarrow 0}\partial_y^+ u(x,y)$ exists,
and this limit extends $\partial_y^+ u$
to a locally bounded function on $\bfR^d\times[0,y_1)$,
which is c\`adl\`ag in $y$, for any fixed $x\in\bfR^d$.

If both $y_1<\infty$ $\&$ $\hat{m}([0,y_1))<+\infty$,
then the limit
$\partial_y^- u_f(x,y_1-)=\lim_{y\uparrow y_1}\partial_y^+ u(x,y)$ 
also exists, extending $\partial_y^- u$
to a locally bounded function on $\bfR^d\times(0,y_1]$,
which is c\`agl\`ad in $y$, for any fixed $x\in\bfR^d$.
\end{corollary}
\begin{remark}\rm
This corollary is fairly standard, so we only sketch its proof in the Appendix.
This sketch also makes clear that, once all $x$-direction second order partial derivatives
$\partial_{ij}u,\,1\le i,j\le d$, are locally bounded in $\bfR^d\times(0,y_1)$,
then $u,\,\partial_i u,\,1\le i\le d,\,\partial_y^\pm u$ are necessarily 
locally bounded in $\bfR^d\times(0,y_1)$, too.
Assumption \ref{BCxreg} rather puts conditions on how 
the $x$-direction partial derivatives behave near the boundary of $\bfR^d\times(0,y_1)$.
\end{remark}

Now, for fixed $x\in\bfR^d$, 
let $(\Omega,{\cal F},\pp_{\!\!x})$ be a complete probability space big enough
to carry all random variables $X,Y,(B_1,\dots,B_p),B$, as described above, with $X,Y$
starting at $X(0)\stackrel{a.s.}{=}x,\,Y(0)\stackrel{a.s.}{=}0$, respectively.
Moreover, we choose a
suitable\footnote{That is, chosen according to the technical assumptions
the results used throughout this paper rely upon.}
filtration, all processes are adapted to, and all stopping times refer to.
The following stopping times,
\[
\tau_r(Y)\,\stackrel{\rm def}{=}\,\inf\{t\ge 0:Y(t)=r\},
\] 
with respect to levels $r\ge 0$, will be frequently used, 
with respect to $Y$, but also with respect to other one-dimensional processes,
for example $\tau_r(\|X\|)$ with respect to $\|X\|$ etc.
\begin{lemma}\label{ito}
Let $u_f$ be the extension of a solution $u$ to (\ref{PDE}) 
satisfying Assumption \ref{BCxreg}.	

(a) If $y_1=\infty$,
\begin{align*}
&u_f(X(t),Y(t))-f(x)
\,=
\int_0^t\partial_y^+ u_f(X(s),0)\,\dd L_s^0(Y)
\,+
\int_0^t{\cal L}_x u_f(X(s),0)\,{\bf 1}_{\{0\}}(Y(s))\,\dd s\\
+&
\sum_{i=1}^d\sum_{k=1}^p\int_0^t\partial_i u_f(X(s),Y(s))\,\sigma_{ik}(X(s))\,\dd B_k(s)
\,+
\sqrt{2}\int_0^t\partial_y^+ u_f(X(s),Y(s))\,b(Y(s))\,\dd B(s),
\end{align*}
for all $t\ge 0$, a.s.

(b) If $y_1<\infty$, and $Y$ is absorbing at $y_1$,
\begin{align*}
&\,u_f(X(t),Y(t))\,{\bf 1}_{\{t<\tau_{y_1}(Y)\}}
+u_f(X(t),y_1-)\,{\bf 1}_{\{t\geq\tau_{y_1}(Y)\}}
-f(x)\\
=\,&
\int_0^t\partial_y^+ u_f(X(s),0)\,\dd L_s^0(Y)\\
+\,&
\int_0^t{\cal L}_x u_f(X(s),0)\,{\bf 1}_{\{0\}}(Y(s))\,\dd s
\,+
\int_0^t{\bf 1}_{\{s\geq\tau_{y_1}(Y)\}}\,{\cal L}_x u_f(X(s),y_1-)\,\dd s\\
+\,&
\sum_{i=1}^d\sum_{k=1}^p\int_0^t{\bf 1}_{\{s<\tau_{y_1}(Y)\}}\,
\partial_i u_f(X(s),Y(s))\,\sigma_{ik}(X(s))\,\dd B_k(s)\\
+\,&
\sum_{i=1}^d\sum_{k=1}^p\int_0^t{\bf 1}_{\{s\geq\tau_{y_1}(Y)\}}\,
\partial_i u_f(X(s),y_1-)\,\sigma_{ik}(X(s))\,\dd B_k(s)\\
+\,&
\sqrt{2}\int_0^t{\bf 1}_{\{s<\tau_{y_1}(Y)\}}\,
\partial_y^+ u_f(X(s),Y(s))\,b(Y(s))\,\dd B(s),
\end{align*}
for all $t\ge 0$, a.s.

(c) If $y_1<\infty$, and $Y$ is not absorbing at $y_1$,
and $\partial_y^- u_f(\cdot,y_1-)$ is a continuous function on $\bfR^d$,
\begin{align*}
&\,u_f(X(t),Y(t))\,{\bf 1}_{\{Y(t)<y_1\}}
+u_f(X(t),y_1-)\,{\bf 1}_{\{y_1\}}(Y(t))
-f(x)\\
=\,&
\int_0^t\partial_y^+ u_f(X(s),0)\,\dd L_s^0(Y) - \int_0^t\partial_y^- u_f(X(s),y_1-)\,\dd L_s^{y_1}(Y)\\
+\,&
\int_0^t{\cal L}_x u_f(X(s),0)\,{\bf 1}_{\{0\}}(Y(s))\,\dd s
\,+
\int_0^t{\cal L}_x u_f(X(s),y_1-)\,{\bf 1}_{\{y_1\}}(Y(s))\,\dd s\\
+\,&
\sum_{i=1}^d\sum_{k=1}^p\int_0^t\partial_i u_f(X(s),Y(s))\,\sigma_{ik}(X(s))\,{\bf 1}_{\{Y(s)<y_1\}}\,\dd B_k(s)\\
+\,&
\sum_{i=1}^d\sum_{k=1}^p\int_0^t\partial_i u_f(X(s),y_1-)\,\sigma_{ik}(X(s))\,{\bf 1}_{\{y_1\}}(Y(s))\,\dd B_k(s)\\
+\,&
\sqrt{2}\int_0^t\partial_y^+ u_f(X(s),Y(s))\,b(Y(s))\,\dd B(s),
\end{align*}
for all $t\ge 0$, a.s.
\end{lemma}
\begin{remark}\rm\label{goody}
If $y_1<\infty$ is not absorbing,  
then (S2) on page \pageref{speedDefi} states ${m}([0,y_1])<+\infty$,
and hence $\hat{m}([0,y_1])<+\infty$, 
by simply taking into account how $m$ was defined in terms of $\hat{m}$.
Thus, the representation of
$\partial_y^- u_f(\cdot,y_1-)$ given in the proof of Corollary \ref{BCyreg}
can be used to show the implication: 
if $\partial_y^- u_f(\cdot,y_1-)$ is a continuous function on $\bfR^d$,
then $\partial_y^+ u(\cdot,y^\star)$ would be one, too,
for any interior value $y^\star\in(0,y_1)$.
Indeed, this implication follows by dominated convergence
combining Assumption \ref{BCxreg} and $u(\cdot,y)\in C^2(\bfR^d),\,y\in(0,y_1)$.
Vice versa, if $\partial_y^+ u(\cdot,y^\star)$ was continuous,
for some $y^\star\in(0,y_1)$, then $\partial_y^- u_f(\cdot,y_1-)$ would be, too.
All in all, stating $\partial_y^- u_f(\cdot,y_1-)\in C(\bfR^d)$
is equivalent to stating $\partial_y^+ u(\cdot,y)\in C(\bfR^d)$,
for all $y\in(0,y_1)$. Moreover, by similar arguments, 
$\partial_y^+ u_f(\cdot,0)\in C(\bfR^d)$ is also equivalent to
$\partial_y^+ u(\cdot,y)\in C(\bfR^d)$, for all $y\in(0,y_1)$,
and hence $\partial_y^+ u_f(\cdot,0)\in C(\bfR^d)$ implies
$\partial_y^- u_f(\cdot,y_1-)\in C(\bfR^d)$, in particular.
\end{remark}

Next, observe that the pair of random variables $(X,Y)$ describes a 
stochastic process on $(\Omega,{\cal F},\pp_{\!\!x})$
taking values in 
$\bfR^d\times\left(\rule{0pt}{12pt}\right.[0,y_1]\cap\bfR\left.\rule{0pt}{12pt}\right)$.
This process is associated with a so-called {\it trace process},
$Z=[Z(t),\,t\ge 0]$, which is the trace of the process $(X,Y)$ 
when touching the hyperplane 
$\{(x,0):x\in\bfR\}
\subseteq
\bfR^d\times\left(\rule{0pt}{12pt}\right.[0,y_1]\cap\bfR\left.\rule{0pt}{12pt}\right)$,
and which can be given by
\begin{equation}\label{defiZ}
Z(t)\,=\,X(T_t),\quad t\ge 0,
\end{equation}
where $[T_t,\,t\ge 0]$ denotes the right-inverse of the symmetric local time
$[L^0_t(Y),\,t\ge 0]$ of Y at zero. 

Of course, if $L^0_{\infty}(Y) = \lim_{t \to \infty} L^0_t(Y) < +\infty$,
then $T_t=+\infty$, for all $t\ge L^0_{\infty}(Y)$,
and $Z$ would be a killed process 
with finite life time $L^0_{\infty}(Y)$. We denote its cemetery-state by $\dagger$,
which is added to $\bfR^d$ in the usual way,
and we also define $X(\infty)=\dagger$, to be consistent with (\ref{defiZ}).
Any function $f:\bfR^d\to\bfR$ is also considered a function
$f:\bfR^d\cup\{\dagger\}\to\bfR$ by putting $f(\dagger)=0$.

The purpose of Lemma \ref{ito} is to be able to work out the limit of
$\ee_x[f(Z(t)) - f(x)]/t$, when $t\downarrow 0$,
for any fixed $x\in\bfR^d$,
and any bounded function $f:\bfR^d\to\bfR$ of a certain degree of regularity.
For the corresponding result, which is the main result of this paper,
we have to differ between infinite and finite $y_1$, but also between 
$y_1$ is absorbing and not-absorbing in the case of finite $y_1$.
Moreover, because of 
$\pp_{\!\!x}(\{\tau_{y_1}(Y)<+\infty\})\in\{0,1\}$,
see \cite[Lemma (2.9)]{AS1998} for example,
the case of $y_1$ is absorbing splits into two further cases.

We first describe the above cases analytically in terms of the Krein string
defining $Y$, and then we state the main theorem.
Note that (S4) is only needed for the theorem, not the lemma,
because only the proof of the theorem requires the semimartingale decomposition of $Y$
given by (\ref{SDE-Y}).
Also, (S4) obviously implies that $\hat{m}$ is strictly increasing on $(0,y_1)$.
\begin{lemma}\label{analytic for Y}
Choose a non-trivial Krein string $\hat{m}$ which is strictly increasing on $(0,y_1)$,
and construct the diffusion $Y$ as in the proof of Lemma A.1 in the Appendix.
Assume $y_1<\infty$. Then:
\begin{itemize}
\item[(i)]
$Y$ is absorbing at $y_1$ if and only if $\hat{m}([0,y_1])=+\infty$;
\item[(ii)]
$\tau_{y_1}(Y)\stackrel{a.s.}{=}+\infty$ if and only if
$\int_{[0,y_1)}(y_1-y)\,\hat{m}(\dd y)=+\infty$.
\end{itemize}
\end{lemma}
\begin{theorem}\label{trace_gen_theorem}
Choose a non-trivial Krein string $\hat{m}$ satisfying (S4),
and select a PDE operator ${\cal L}_x$ 
according to (\ref{formal generator}) and Remark \ref{implicit conditions}.
Let $u$ be a bounded solution to (\ref{PDE}) satisfying Assumption \ref{BCxreg},
and suppose that the extension $u_f$ satisfies
$\partial^+_y u_f(\cdot, 0) 
+
m_0\,\mathcal{L}_x u_f(\cdot, 0)\in C_b(\mathbb{R}^d)$.
Consider one of the following cases:
\begin{enumerate}
\item[(a)] $y_1 = \infty$;
\item[(b1)] $y_1 < \infty$, 
$\hat{m}([0,y_1])=+\infty$ but $\int_{[0,y_1)}(y_1-y)\,\hat{m}(\dd y)<+\infty$,
$u_f(\cdot, y_1- ) \equiv 0$,
and the extension $u_f:\bfR^d\times[0,y_1]\to\bfR$
is jointly continuous at any $(x,y_1),\,x\in\bfR^d$;
\item[(b2)]
$y_1 < \infty$, $\int_{[0,y_1)}(y_1-y)\,\hat{m}(\dd y)=+\infty$,
and\, $\sup_{x\in\bfR^d}|u(x,y_1-h)|\to 0,\,h\downarrow 0$;
\item[(c)] $y_1 < \infty$, $\hat{m}([0,y_1])<+\infty$, 
and $u_f$ satisfies 
$\partial_y^- u_f(\cdot,y_1-) \equiv m_1\,{\cal L}_x u_f(\cdot,y_1-)\in C(\bfR^d)$.
\end{enumerate}
Then, for any $x\in\bfR^d$,
\begin{align*}
\lim_{t \downarrow 0}\,\frac{1}{t}\,\ee_x[\,f(Z(t)) - f(x)\,]
\,=\, 
\partial_y^+u_f(x, 0) + m_0\,\mathcal{L}_x u_f(x, 0).
\end{align*}	
\end{theorem}

The question arises whether the representation of 
$\lim_{t\downarrow 0}\ee_x[f(Z(t)) - f(x)]/t$
given in the above theorem 
is related to the generator of a semigroup associated with $Z$.

To start with, under the made assumptions,
both processes $X$ and $Y$ are strong Markov processes,
the process $X$ because of \cite[Theorem 5.4.20]{KS1991},
and $Y$ by construction in \cite[Chapter VI]{AS1998}.

Of course, the transition semigroup associated with $X$ 
is a semigroup of contractions acting on bounded measurable functions.
Since the coefficients of the SDE (\ref{SDE}) are supposed to be continuous,
it is easy to see that the transition semigroup  preserves 
the Banach space $C_b(\bfR^d)$ of bounded continuous functions.
However, as we allow unbounded coefficients, 
its restriction to $C_b(\bfR^d)$ might not be strongly continuous with respect to the uniform norm. 

So, for the next corollary, 
let us assume that there exists a strongly continuous semigroup of contractions
on a Banach space $E\subseteq C(\bfR^d)$,
which coincides with the transition semigroup on $C_b(\bfR^d)\cap E$.
Linking $E$ and our PDE-problem,
we further assume that $E$ densely includes 
all twice continuously differentiable functions with compact support.

Then, this semigroup is generated by a closed operator
on $E$ with dense domain. We naturally denote this generator by ${\cal L}_x$ 
as it obviously coincides with ${\cal L}_x$ given by (\ref{formal generator})
on the subspace of all twice continuously differentiable functions
with compact support. We also say that the process $X$ is associated with this semigroup
(or with the generator of this semigroup), and we denote the semigroup by
$e^{t{\cal L}_x},\,t\ge 0$, in what follows, slightly abusing notation
because this notation is usually reserved for semigroups generated
by self-adjoint operators on Hilbert spaces.

Next, by the strong Markov property of $Y$,
the process $[T_t, t \geq 0]$ is a (possibly killed) subordinator,
see Section 68 of Chapter VIII in \cite{S1988}, for example.
Let $\psi(\lambda),\,\lambda>0$, denote the Laplace exponent
of this subordinator, that is
\begin{equation}\label{exponent}
\ee_x[e^{-\lambda T_t}]
\,=\,
\ee_x[e^{-\lambda T_t}{\bf 1}_{\{T_t<\infty\}}]
\,=\,
\int_{[0,\infty)}e^{-\tau\lambda}\,\gamma_t(\dd\tau)
\,=\,
e^{-t\psi(\lambda)},
\end{equation}
where $\gamma_t$ stands for the law of the random time $T_t$.
Then Phillips' theorem, see \cite{P1952}, 
yields that the family of Stieltjes-type Banach space valued integrals
\begin{equation}\label{stieltjes}
\int_{[0,\infty)}e^{\tau{\cal L}_x}\,\gamma_t(\dd\tau),
\quad t\ge 0,
\end{equation}
defines a strongly continuous contraction semigroup on $E$,
whose generator has a domain including the domain of ${\cal L}_x$.
This semigroup obviously coincides with the transition semigroup of $Z$
on $C_b(\bfR^d)\cap E$ because, for $f\in C_b(\bfR^d)\cap E$,
\[
\ee_x[f(Z(t))]
\,=\,
\ee_x[f(X(T_t))]
\,=\,
\int_{[0,\infty)}[e^{\tau{\cal L}_x}f](x)\,\gamma_t(\dd\tau)
\,=\,
[\int_{[0,\infty)}e^{\tau{\cal L}_x}\,\gamma_t(\dd\tau)f\,](x),
\]
and hence $Z$ is a strong Markov process which is associated with
the generator of the strongly continuous contraction semigroup
$\int_{[0,\infty)}e^{\tau{\cal L}_x}\,\gamma_t(\dd\tau),\,t\ge 0$,
on $E$.

Even more, formally substituting $e^{-\tau\lambda}$ by $e^{-\tau(-{\cal L}_x)}$ in (\ref{exponent}),
gives
\[
\int_{[0,\infty)}e^{\tau{\cal L}_x}\,\gamma_t(\dd\tau)
\,=\,
e^{-t\psi(-{\cal L}_x)},
\quad t\ge 0,
\]
suggesting that, formally, the generator of the semigroup 
$\int_{[0,\infty)}e^{\tau{\cal L}_x}\,\gamma_t(\dd\tau),\,t\ge 0$,
that is a generator associated with $Z$, 
could be identified with $-\psi(-{\cal L}_x)$, in some sense.
This functional calculus based on subordination has indeed been established,
see \cite{S1998} for a good account on this theory.
Furthermore, if ${\cal L}_x$ was a self-adjoint operator on a Hilbert space,
then the functional calculus for self-adjoint operators on a Hilbert space
would produce the same operator $-\psi(-{\cal L}_x)$,
assuming that $\psi$ is an admissible function in this calculus, too.

All in all, the well-established operator $-\psi(-{\cal L}_x)$ on $E$,
where $\psi$ is the Laplace exponent of the subordinator $[T_t,\,t\ge 0]$,
plays the role of a generator associated with the trace process $Z$,
and the domain of ${\cal L}_x$ is even a core for $-\psi(-{\cal L}_x)$.
Therefore,
\[
\lim_{t\downarrow 0}\,
\Big\|\,
\frac{1}{t}\,\ee_{\,\boldsymbol{\cdot}}[\,f(Z(t))-f(\boldsymbol{\cdot})\,]+\psi(-{\cal L}_x)f
\,\Big\|_{E}
\,=\,
0,
\]
for any bounded $f$ in the domain of ${\cal L}_x$,
and since uniform convergence implies pointwise convergence,
we have got the following answer to our question below Theorem \ref{trace_gen_theorem}.
\begin{corollary}\label{ourD2N}
Under the conditions of Theorem \ref{trace_gen_theorem},
if $X$ is associated with a strongly continuous semigroup of contractions, $e^{t{\cal L}_x},\,t\ge 0$,
on a Banach space $E\subseteq C(\bfR^d)$,
and if $u_f(\cdot,0)$ is in the domain of the generator of this semigroup,
then
\[
-\psi(-{\cal L}_x)f(x)
\,=\,
\partial_y^+u_f(x, 0) + m_0\,\mathcal{L}_x u_f(x, 0),
\]
for any $x\in\bfR^d$.
\end{corollary}
\begin{remark}\rm \label{DtoNtrue}
(a) Theorem \ref{trace_gen_theorem} only asserts pointwise convergence
of $\ee_x[\,f(Z(t))-f(x)\,]/t$
which is much weaker than convergence
in a Banach space (note that pointwise convergence is not even metrizable).
Thus, if the conditions of Theorem \ref{trace_gen_theorem}
are satisfied, then $u_f(\cdot,0)$ does \underline{not} necessarily
have to be in the domain of a generator associated with the trace process.
However, if $E$ can be chosen to be $C_0(\bfR^d)$,
that is the space of continuous functions vanishing at infinity,
and if $-\psi(-{\cal L}_x)$ satisfies the positive maximum principle,
then pointwise convergence would even imply uniform convergence, once both 
$u_f(\cdot,0)$ and $\partial_y^+u_f(\cdot, 0) + m_0\,\mathcal{L}_x u_f(\cdot, 0)$
are in $C_0(\bfR^d)$.

(b) By standard arguments,
the Banach space $E$ can be chosen to be $C_0(\bfR^d)$
if the coefficients of ${\cal L}_x$ are bounded.
But, in the general unbounded case, the semigroup $e^{t{\cal L}_x},\,t\ge 0$, 
might not preserve this space and a growth condition at infinity would be needed.
For example, if there exists a twice continuously differentiable  Lyapunov-type function 
$f_L:\bfR^d\to(0,\infty)$ such that
$\lim_{x\to\infty}f_L(x)=+\infty$ and ${\cal L}_x f_L\le 0$,
then the choice
$
E
\,=\,
\{f\in C(\bfR^d):\mbox{$\lim_{x\to\infty}$}f(x)/f_L(x)=0\}
$
would work. Here, ${\cal L}_x f_L\le 0$ is strong enough to ensure that $e^{t{\cal L}_x},\,t\ge 0$, 
is a strongly continuous semigroup of contractions on $E$,
while the weaker condition ${\cal L}_x f_L\le const(f_L+1)$ would only
ensure strong continuity. 
Note that linear growth, (\ref{linearGrowth}), of the coefficients of ${\cal L}_x$ at infinity is only sufficient
for the weaker condition, and that without contraction property the integrals
in (\ref{stieltjes}) might not converge. 
Thus, even in the simple case of unbounded coefficients with linear growth,
it might only be possible to justify the convergence of
$\ee_{\,\boldsymbol{\cdot}}[\,f(Z(t))-f(\boldsymbol{\cdot})\,]/t$
with respect to a topology weaker than the canonical topology of the underlying
Banach space, so it might be hard to associate $Z$ with a `proper' generator.

(c) Some generators ${\cal L}_x$, for example $\Delta_x$,
can also be considered in $L^p(\bfR^d)$-spaces, $1\le p<\infty$,
determined by a reference measure, which in case of $\Delta_x$ 
would simply be the Lebesgue measure on $\bfR^d$. 
Then, under the conditions of Theorem \ref{trace_gen_theorem},
if $u_f(\cdot,0)$ is in the $L^p(\bfR^d)$-domain of the generator ${\cal L}_x$,
the assertion of Corollary \ref{ourD2N}
would only hold true for almost every $x\in\bfR^d$
with respect to the reference measure, by obvious reasons.
Of course, this assertion would nevertheless imply the identity
\[
-\psi(-{\cal L}_x)f
\,=\,
\partial_y^+u_f(\cdot, 0) + m_0\,\mathcal{L}_x u_f(\cdot, 0),
\]
in the corresponding $L^p$-space, though only the sum on the right-hand side,
not each summand, has to be in that $L^p$-space.

(d) Subject to $m_0=0$,
the conditions of this corollary translate into conditions on
${\cal L}_x,\,\hat{m}$, and the Dirichlet data, such that 
the operator $-\psi(-{\cal L}_x)$ acts as Dirichlet-to-Neumann map, too. 
For how these conditions compare to existing conditions, 
and for how the Laplace exponent $\psi$ depends on $\hat{m}$, 
the reader is referred to Section 4.
\end{remark}

So far, we have considered $(X,Y)$ a stochastic process on some
probability space $(\Omega,{\cal F},\pp_{\!\!x})$,
taking values in 
$\bfR^d\times\left(\rule{0pt}{12pt}\right.[0,y_1]\cap\bfR\left.\rule{0pt}{12pt}\right)$,
with $X$ and $Y$ almost surely starting at $x$ and $0$, respectively. 
The starting point $Y(0)\stackrel{a.s.}{=}0$ 
was crucial for showing that an operator associated with the trace process 
maps the Dirichlet boundary condition of a solution of (\ref{PDE}) 
to its Neumann boundary condition.

But, 
we can also apply our It\^o's lemma \ref{ito} to obtain stochastic representations
for fundamental solutions of our PDE (\ref{PDE})
with respect to different boundary conditions.
For this purpose, we need the processes $X$ and $Y$ 
starting at $x\in\bfR^d$ and arbitrary $y\in(0,y_1)$, respectively,
and we are going to write $X^x$ and $Y^y$ instead of $X$ and $Y$
in the remaining part of this section,
in order to emphasise the specific choice of starting values.
Also, slightly abusing notation, we will denote the underlying probability measure
by $\pp$ dropping the sub-index $x$, though $\pp$ could in principle
be different for different $(x,y)\in\bfR^d\times(0,y_1)$.

Now, for arbitrary but fixed $(x,y)\in\bfR^d\times(0,y_1)$, 
let $\pi(x,y,\dd x')$ denote the probability law determined by
\[
\int_{\bfR^d}\pi(x,y,\dd x')f(x')\,=
\left\{\begin{array}{ccc}
\ee[\,f(X^x(\tau_0(Y^y))){\bf 1}_{\{\tau_0(Y^y)<\tau_{y_1}(Y^y)\}}\,]
&:&\mbox{$y_1<\infty$ absorbing},\\
\rule{0pt}{12pt}
\ee[\,f(X^x(\tau_0(Y^y)))\,]&:&\mbox{else},
\end{array}\right.
\]
where $f:\bfR^d\to\bfR$ can be any bounded measurable function.
\begin{corollary}\label{fundamental solution}
For any bounded solution $u$ of (\ref{PDE}) satisfying Assumption \ref{BCxreg},
if the extension 
$u_f:
\bfR^d\times\left(\rule{0pt}{12pt}\right.[0,y_1]\cap\bfR\left.\rule{0pt}{12pt}\right)
\to\bfR$
is jointly continuous at any $(x,0),\,x\in\bfR^d$,
then, in each of the four cases considered in Theorem \ref{trace_gen_theorem},
\[
u(x,y)\,=\int_{\bfR^d}\pi(x,y,\dd x')f(x'),
\quad(x,y)\in\bfR^d\times(0,y_1),
\]
and hence $\pi(x,y,\dd x')$ acts as a fundamental solution
in each of theses cases.
\end{corollary}
\section{Proofs}
\begin{proof}[Proof of \textbf{Lemma \ref{ito}}]
The following proof of (a) demonstrates the main idea used in all three cases,
and hence we will restrict ourselves to what is different to this proof when proving (b),(c).

(a) To start with, we `mollify' $u_f$ introducing
\[
u_f^\varepsilon(x,y)
\,\stackrel{\mbox{\tiny def}}{=}
\int_0^\infty\varrho_\varepsilon(y-y')\,u_f(x,y')\,\dd y',
\quad(x,y)\in\bfR^d\times[0,\infty),
\quad\varepsilon>0, 
\]
where $\varrho_\varepsilon(y)=\varrho(y/\varepsilon)/\varepsilon,\,y\in\bfR$,
using a `right-hand' mollifier $\varrho\in C^2(\bfR)$ with compact support
in $(-1,0)$ satisfying $\varrho\ge 0$ and $\int_\bfR\varrho(y)\,\dd y=1$. 

Note that, because of Assumption \ref{BCxreg}, on the one hand, and because 
the support of $\varrho$ is bounded away from zero, on the other,
the mollified solution $u_f^\varepsilon$ is an element of $C^2(\bfR^d\times[0,\infty))$.

Recall that $u_f$ is not necessarily jointly continuous, so we do not know if
$[u_f(X(t),Y(t)),\,t\ge 0]$ is a continuous stochastic process. However,
\[
u_f(X(\omega,t),Y(\omega,t))
\,=\,
\mbox{$\lim_{\varepsilon\downarrow 0}$}\,u_f^\varepsilon(X(\omega,t),Y(\omega,t)),
\quad\forall\,(\omega,t)\in\Omega\times[0,\infty),
\]
where $[u_f^\varepsilon(X(t),Y(t)),\,t\ge 0]$ is of course a continuous stochastic process.
Therefore, the process $[u_f(X(t),Y(t)),\,t\ge 0]$ is at least predictable,
and hence $[\mathfrak{Z}(t),\,t\ge 0]$ defined by
\begin{align*}
\mathfrak{Z}(t)
&\,\stackrel{\mbox{\tiny def}}{=}\,
u_f(X(t),Y(t))-f(x)\\
&\,-
\int_0^t\partial_y^+ u_f(X(s),0)\,\dd L_s^0(Y)
\,-
\int_0^t{\cal L}_x u_f(X(s),0)\,{\bf 1}_{\{0\}}(Y(s))\,\dd s\\
&\,-
\sum_{i=1}^d\sum_{k=1}^p\int_0^t\partial_i u_f(X(s),Y(s))\,\sigma_{ik}(X(s))\,\dd B_k(s)
\,-
\sqrt{2}\int_0^t\partial_y^+ u_f(X(s),Y(s))\,b(Y(s))\,\dd B(s)
\end{align*}
is predictable, too, because the stochastic integrals which are well-defined
by (\ref{SDE}), (\ref{SDE-Y}), and Assumption \ref{BCxreg},
give continuous processes in $t$.
Thus, using \cite[Prop.\ I 2.18 b)]{JS2003}, to prove part (a) of the lemma,
it is sufficient to show that $\pp_{\!\!x}(\{\mathfrak{Z}(t\wedge\tau)=0\})=1$,
for any $t\ge 0$, and any predictable stopping time $\tau$.

Next, set $\tau_N=\tau_N(\|X\|)\wedge N$,
and choose a sequence $S_N,\,N=1,2,\dots$, of finite stopping times
which is $L^2(\pp_{\!\!x})$-localising for the local martingale
$[\int_0^t b(Y(s))\,\dd B(s),\,t\ge 0]$. 
Since $\tau_N\wedge S_N\wedge\tau_N(Y)$ grows to infinity,
when $N\to\infty$,
it is enough to prove that
$\pp_{\!\!x}(\{\mathfrak{Z}(t\wedge\tau_N\wedge S_N\wedge\tau_N(Y)\wedge\tau)=0\})=1$,
for any $t\ge 0$, any $N=1,2,\dots$, and any predictable stopping time $\tau$.
Also, by Assumption \ref{BCxreg}, the sequence
$\tau_N\wedge S_N\wedge\tau_N(Y),\,N=1,2,\dots$, would be $L^2(\pp_{\!\!x})$-localising
for all local martingales used in the definition of $[\mathfrak{Z}(t),\,t\ge 0]$.

So, fix $t\ge 0,\,N\ge 1$, and a predictable stopping time $\tau$.
For  $t_N=t\wedge\tau_N\wedge S_N\wedge\tau_N(Y)\wedge\tau$, we are going to show
that $\pp_{\!\!x}(\{\mathfrak{Z}(t_N)=0\})=1$.
 
First, for any $\varepsilon>0$,
the classical version of It\^o's lemma gives
\begin{align}\label{leftHand}
&\hspace{2cm}u_f^\varepsilon(X({t_N}),Y({t_N}))-u_f^\varepsilon(x,0)
\,-
\int_0^{t_N}\partial_y u_f^\varepsilon(X(s),0)\,\dd L_s^0(Y)\nonumber\\
-&
\sum_{i=1}^d\sum_{k=1}^p\int_0^{t_N}\partial_i u_f^\varepsilon(X(s),Y(s))\,\sigma_{ik}(X(s))\,\dd B_k(s)
\,-\,
\sqrt{2}\int_0^{t_N}\partial_y u_f^\varepsilon(X(s),Y(s))\,b(Y(s))\,\dd B(s)\nonumber\\
&\hspace{1.2cm}\,\stackrel{a.s.}{=}
\int_0^{t_N}{\cal L}_x u_f^\varepsilon(X(s),Y(s))\,\dd s
\,+
\int_0^{t_N}\partial_y^2 u_f^\varepsilon(X(s),Y(s))\,b^2(Y(s))\,\dd s.
\end{align}

We claim that, for any chosen sequence of $\varepsilon$-values converging to zero,
there is a subsequence $(\varepsilon_n)_{n=1}^\infty$ such that the above
equation's left-hand side almost surely converges to 
\begin{align*}
&\hspace{2cm}u_f(X({t_N}),Y({t_N}))-f(x)
\,-
\int_0^{t_N}\partial_y^+ u_f(X(s),0)\,\dd L_s^0(Y)\\
-&
\sum_{i=1}^d\sum_{k=1}^p\int_0^{t_N}\partial_i u_f(X(s),Y(s))\,\sigma_{ik}(X(s))\,\dd B_k(s)
\,-\,
\sqrt{2}\int_0^{t_N}\partial_y^+ u_f(X(s),Y(s))\,b(Y(s))\,\dd B(s),
\end{align*}
when $n\to\infty$.

Indeed, the limit of 
$u_f^\varepsilon(X({t_N}),Y({t_N}))-u_f^\varepsilon(x,0),\,\varepsilon\downarrow 0$,
is obvious, by Remark \ref{afterBCxreg}.

Next, for $0\le s\le t_N$, partial integration yields
\[
\partial_y u_f^\varepsilon(X(s),Y(s))
\,=\,
\int_0^\infty\varrho_\varepsilon'(Y(s)-y')\,u_f(X(s),y')\,\dd y'
\,=\,
\int_{-\infty}^0\varrho_\varepsilon(y')\,\partial_y^+ u_f(X(s),Y(s)-y')\,\dd y',
\]
which converges to $\partial_y^+ u_f(X(s),Y(s))$, when $\varepsilon\downarrow 0$.
Furthermore,
\[
\sup_{0\le s\le t_N}|\partial_y u_f^\varepsilon(X(s),Y(s))|
\,\le\,
\hspace{-15pt}
\renewcommand{\arraystretch}{0.7}
\sup_{\begin{array}{c}
\scriptstyle \|x\|\le N\\
\scriptstyle y\in[0,N+\varepsilon]
\end{array}}
\renewcommand{\arraystretch}{1.0}
\hspace{-15pt}
|\partial_y^+ u_f(x,y)|
\underbrace{\int\varrho_\varepsilon(y')\,\dd y'}_{=\,1},
\quad\forall\,\varepsilon>0,
\]
where, by Corollary \ref{BCyreg}, the supremum on the right-hand side
is uniformly bounded in $\varepsilon$,
for any chosen sequence of $\varepsilon$-values converging to zero.
Therefore,
\begin{align*}
\int_0^{t_N}\partial_y u_f^\varepsilon(X(s),0)\,\dd L_s^0(Y)
&=
\int_0^{t_N}\partial_y u_f^\varepsilon(X(s),Y(s))\,\dd L_s^0(Y)\\
&\stackrel{\varepsilon\downarrow 0}{\longrightarrow}
\int_0^{t_N}\partial_y^+ u_f(X(s),Y(s))\,\dd L_s^0(Y)
\,=\,
\int_0^{t_N}\partial_y^+ u_f(X(s),0)\,\dd L_s^0(Y)
\end{align*}
follows by dominated convergence.

Lastly, it is not hard to see that all stochastic integrals 
on the left-hand side of (\ref{leftHand})
converge duly in $L^2(\pp_{\!\!x})$, when $\varepsilon\downarrow 0$, 
proving the claim we made above. We only discuss the case of
$\int_0^{t_N}\partial_y u_f^\varepsilon(X(s),Y(s))\,b(Y(s))\,\dd B(s)$---the 
other cases are easier because the coefficients $\sigma_{ik}$
are supposed to be continuous. Since $t_N\le\tau_N\wedge S_N\wedge\tau_N(Y)$,
we obtain that
\begin{align*}
&\,\ee_x\left[
\left(\rule{0pt}{12pt}\right.
\int_0^{t_N}\partial_y u_f^\varepsilon(X(s),Y(s))\,b(Y(s))\,\dd B(s)
\,-
\int_0^{t_N}\partial_y^+ u_f(X(s),Y(s))\,b(Y(s))\,\dd B(s)
\left.\rule{0pt}{12pt}\right)^2
\right]\\
=&\,
\ee_x\left[
\int_0^{t_N}
[\,\partial_y u_f^\varepsilon(X(s),Y(s))-\partial_y^+ u_f(X(s),Y(s))\,]^2
\,b^2(Y(s))\,\dd s
\right]\!,
\end{align*}
where $\ee_x[\int_0^{t_N}b^2(Y(s))\,\dd s\,]<+\infty$.
Then, again by dominated convergence,
the $\varepsilon$-limit of the right-hand side can be taken with respect to the integrand,
and this limit is zero, for all $s\in[0,t_N]$.

So, for $\pp_{\!\!x}(\{\mathfrak{Z}(t_N)=0\})=1$, the remaining step is to show that
\begin{equation}\label{remains}
\int_0^{t_N}{\cal L}_x u_f^\varepsilon(X(s),Y(s))\,\dd s
\,+
\int_0^{t_N}\partial_y^2 u_f^\varepsilon(X(s),Y(s))\,b^2(Y(s))\,\dd s
\end{equation}
almost surely converges to
\[
\int_0^{t_N}{\cal L}_x u_f(X(s),0)\,{\bf 1}_{\{0\}}(Y(s))\,\dd s,
\]
when $\varepsilon\downarrow 0$.

To see this, we use the construction of $[Y(t),\,t\ge 0]$ 
given in the proof of Lemma A.1 in the Appendix.
Also, recall the Lebesgue decomposition of the measures $m$ and $\hat{m}$,
where $m$ has got the properties (S1-3). Of course, the measure $\hat{m}$
satisfies the same properties, 
and (S1) implies that $b^{-2}$ is locally integrable on $[0,y_1)$.

Thus, we assume that $Y(t)=W(A_t^{-1}),\,t\ge 0$, where $[W(t),\,t\ge 0]$
is a standard one-dimensional Wiener process, 
given on $(\Omega,{\cal F},\pp_{\!\!x})$, and
\[
A_t\,=\,
\int_{\bfR}\frac{1}{2}\,L_t^y(W)\,\hat{m}(\dd y)
\,=\,
\int_{[0,\infty)}\frac{1}{2}\,L_t^y(W)\,\hat{m}(\dd y),\quad t\ge 0.
\]

Note that $A_\infty=+\infty$, a.s., as stated in \cite[Lemma (6.15)]{AS1998},
and $t\mapsto A_t$ is continuous because the measure $\hat{m}$
is finite on compact subsets of $[0,\infty)$.
Therefore, $s\,=\,A_{A^{-1}_s}$, for all $s\ge 0$, a.s., and time change yields
\begin{align*}
\int_0^{t_N}\partial_y^2 u_f^\varepsilon&(X(s),Y(s))\,b^2(Y(s))\,\dd s
\,\stackrel{a.s.}{=}
\int_0^{A_{A^{-1}_{t_N}}}\partial_y^2 u_f^\varepsilon(X(A_{A^{-1}_s}),W(A^{-1}_s))\,b^2(W(A^{-1}_s))\,\dd s\\
=&
\int_0^{A^{-1}_{t_N}}\partial_y^2 u_f^\varepsilon(X(A_s),W(s))\,b^2(W(s))\,\dd A_s
\end{align*}
which, when applying a kernel-type Fubini theorem, expands into
\begin{align*}
&
\int_{[0,\infty)}
\int_0^{A^{-1}_{t_N}}\partial_y^2 u_f^\varepsilon(X(A_s),W(s))\,\dd L^y_s(W)
\,b^2(y)\,[\frac{1}{2}b^{-2}(y)\,\dd y+\frac{m_0}{2}\delta_0(\dd y)+\frac{1}{2}n(\dd y)]\\
\stackrel{a.s.}{=}&
\int_{[0,\infty)}
\int_0^{A^{-1}_{t_N}}\partial_y^2 u_f^\varepsilon(X(A_s),y)\,\dd L^y_s(W)
\,b^2(y)\,[\frac{1}{2}b^{-2}(y)\,\dd y+\frac{m_0}{2}\delta_0(\dd y)+\frac{1}{2}n(\dd y)]\\
=&
\int_0^\infty
\int_0^{A^{-1}_{t_N}}\partial_y^2 u_f^\varepsilon(X(A_s),y)
\,\dd L^y_s(W)\,{\bf 1}_{\{b^2>0\}}(y)\,\dd y\times 1/2.
\end{align*}
Here, the second equality holds as, almost surely, for all $y$,
the increasing functions $s\mapsto L^y_s(W)$ only grow when $W(s)$ is at $y$,
while the last equality follows from
$({\cal N}\cup\{0,y_1\})\cap\bfR\subseteq\{b^2=0\}$.

However, since $b^{-2}$ is locally integrable on $[0,y_1)=[0,\infty)$,
the set $\{b^2=0\}$ has Lebesgue measure zero, and hence the last integral becomes
\[
\int_0^\infty
\int_0^{A^{-1}_{t_N}}\partial_y^2 u_f^\varepsilon(X(A_s),y)\,\dd L^y_s(W)\,\dd y\times 1/2
\stackrel{\varepsilon\downarrow 0}{\longrightarrow}
-\int_{(0,\infty)}
\int_0^{A^{-1}_{t_N}}{\cal L}_x u_f(X(A_s),y)\,\dd L^y_s(W)\,\hat{m}(\dd y)\times 1/2,
\]
where the limit can be verified as follows:
first, discretize the Stieltjes integral against local time
subject to a further limit of a sum over discrete time points $s_k$,
and then use (\ref{L dm integral}) to identify the limit of the measures 
\[
\partial_y^2 u_f^\varepsilon(X(A_{s_k}),y)\,\dd y
\,=\,
\dd_y[\partial_y u_f^\varepsilon(X(A_{s_k}),y)]
\,=\,
\int_{-\infty}^0\dd y'\,\varrho_\varepsilon(y')\,
\dd_y[\partial_y^+ u_f(X(A_{s_k}),y-y')],
\]
when $\varepsilon\downarrow 0$.

But,
\begin{align*}
&\hspace{2cm}-\int_{[0,\infty)}
\int_0^{A^{-1}_{t_N}}{\cal L}_x u_f(X(A_s),y)\,\dd L^y_s(W){\bf 1}_{(0,\infty)}(y)\,\frac{\hat{m}}{2}(\dd y)\\
\stackrel{a.s.}{=}&
-\int_0^{A^{-1}_{t_N}}{\cal L}_x u_f(X(A_s),W(s)){\bf 1}_{(0,\infty)}(W(s))\,\dd A_s
\,\stackrel{a.s.}{=}
-\int_0^{t_N}{\cal L}_x u_f(X(s),Y(s)){\bf 1}_{(0,\infty)}(Y(s))\,\dd s,
\end{align*}
which almost surely is the $\varepsilon$-limit of the second summand in (\ref{remains}).
Note the subtle point that the indicator ${\bf 1}_{(0,\infty)}$ is due to the fact
that the PDE only holds in the open half plane.

Finally, as $\int_0^{t_N}{\cal L}_x u_f(X(s),Y(s))\,\dd s$ is
the $\varepsilon$-limit of the first summand in (\ref{remains}), the $\varepsilon$-limit 
of (\ref{remains}) can almost surely be given by
\[
\int_0^{t_N}{\cal L}_x u_f(X(s),Y(s)){\bf 1}_{\{0\}}(Y(s))\,\dd s
\,=\,
\int_0^{t_N}{\cal L}_x u_f(X(s),0)\,{\bf 1}_{\{0\}}(Y(s))\,\dd s,
\]
proving part (a) of the lemma.

(b) Fix $t,\tau_N,S_N,\tau$ as in (a), but define
$t_N=t\wedge\tau_N\wedge S_N\wedge\tau_N(L_{\fatdot}^0(Y))\wedge\tau$.
It is again sufficient to show that 
$\pp_{\!\! x}(\{ \mathfrak{Z}(t_N) = 0 \}) = 1$, where  
\begin{align*}
\mathfrak{Z}(t) 
&\,\stackrel{\mbox{\tiny def}}{=}\,
u_f(X(t),Y(t))\,{\bf 1}_{\{t<\tau_{y_1}(Y)\}}
	+u_f(X(t),y_1-)\,{\bf 1}_{\{t\geq \tau_{y_1}(Y)\}}
	-f(x)\\
	&\,-
	\int_0^t\partial_y^+ u_f(X(s),0)\,\dd L_s^0(Y)\\
	&\,-
	\int_0^t{\cal L}_x u_f(X(s),0)\,{\bf 1}_{\{0\}}(Y(s))\,\dd s
	\,-
	\int_0^t{\bf 1}_{\{s\geq \tau_{y_1}(Y)\}}\,{\cal L}_x u_f(X(s),y_1-)\,\dd s\\
&\,-
\sum_{i=1}^d\sum_{k=1}^p\int_0^t{\bf 1}_{\{s<\tau_{y_1}(Y)\}}\,
\partial_i u_f(X(s),Y(s))\,\sigma_{ik}(X(s))\,\dd B_k(s)\\
&\,-
\sum_{i=1}^d\sum_{k=1}^p\int_0^t{\bf 1}_{\{s\geq \tau_{y_1}(Y)\}}\,
\partial_i u_f(X(s),y_1-)\,\sigma_{ik}(X(s))\,\dd B_k(s)\\
&\,-
\sqrt{2}\int_0^t{\bf 1}_{\{s<\tau_{y_1}(Y)\}}\,
\partial_y^+ u_f(X(s),Y(s))\,b(Y(s))\,\dd B(s).
\end{align*}

Observe that, different to the proof of (a), 
it is technically more demanding to work with the mollified version
of $u_f(x,y)$, when $y$ is close to $y_1<\infty$. 
We therefore choose $h\in(0,y_1)$
and build a function $u_{f,h}$ on the whole half-space by
\[
u_{f,h}(x,y)\,=\left\{\begin{array}{ccc}
u_{f}(x,y)&:&y<y_1-h/2,\\ u_{f}(x,y_1-h/2)&:&y\ge y_1-h/2.
\end{array}\right.
\]

Let $u_{f,h}^\varepsilon$ denote the mollified version of $u_{f,h}$ using the
same mollifier $\varrho$ as in the proof of (a),
and let $Y_h$ denote the process $[Y(t\wedge\tau_{y_1-h}(Y)),\,t\ge 0]$.
Note that, for any such $h$, the stopping time
$\tau_{y_1-h}(Y)$ is almost surely finite, because $m([0,y_1-h])<+\infty$.

We first apply the classical version of It\^o's lemma to 
$u_{f,h}^\varepsilon(X({t_N}),Y_h({t_N}))$
and let $\varepsilon$ go to zero. We then prove our claim by letting $h$ go to zero, too. 

Recall that changing $b$ on a set of Lebesgue measure zero would not change
the law of $Y$, thus it would not change the law of $Y_h$, either.
Without loss of generality, we can therefore assume that $b(y_1-h)=0$.
As a consequence, $Y_h$ satisfies the SDE,
\[
\dd Y_h(t)\,=\,\sqrt{2}\,b(Y_h(t))\,\dd B(t)+\dd L^0_t(Y_h),
\]
and hence applying It\^o's formula to $u_{f,h}^\varepsilon(X({t_N}),Y_h({t_N}))$
gives an equation identical to (\ref{leftHand})
with $u_f^\varepsilon$ and $Y$ replaced by $u_{f,h}^\varepsilon$ and $Y_h$, respectively.

Then,
as in the proof of part (a), for any sequence of $\varepsilon$-values converging to zero, 
there exists a subsequence $(\varepsilon_n)_{n=1}^{\infty}$ 
such that, when $n \to \infty$, the left-hand side of this equation 
converges almost surely to
\begin{equation}\label{itoYh}
\left.
\begin{array}{c}
u_{f}(X({t_N}),Y({t_N\wedge\tau_{y_1-h}(Y)}))-f(x)\\
\rule{0pt}{15pt}
-\int_0^{t_N}{\bf 1}_{\{s<\tau_{y_1-h}(Y)\}}\,
\partial^+_y u_f(X(s),0)\,\dd L_s^0(Y)\\
\rule{0pt}{15pt}
-\sum_{i=1}^d\sum_{k=1}^p\int_0^{t_N}\partial_i u_{f}(X(s),Y(s\wedge\tau_{y_1-h}(Y)))\,
\sigma_{ik}(X(s))\,\dd B_k(s)\\
\rule{0pt}{15pt}
-\,\sqrt{2}\int_0^{t_N}{\bf 1}_{\{s<\tau_{y_1-h}(Y)\}}\,
\partial^+_y u_{f}(X(s),Y(s))\,b(Y(s))\,\dd B(s),
\end{array}
\right\}
\end{equation}
where we have used that $u_f=u_{f,h}$, on $\bfR^d\times[0,y_1-h/2)$,
that $L_s^0(Y_h)$ is constant, for $s\ge\tau_{y_1-h}(Y)$,
and that $b(y_1-h)=0$.

The next step is to find the $\varepsilon$-limit of the right-hand side, i.e.
\begin{equation}\label{toFind}
\int_0^{t_N}{\cal L}_x u_{f,h}^\varepsilon(X(s),Y_h(s))\,\dd s
\,+
\int_0^{t_N}\partial_y^2 u_{f,h}^\varepsilon(X(s),Y_h(s))\,b^2(Y_h(s))\,\dd s,
\end{equation}
the second integral of which can be written as
\ 
\[
\int_0^{t_N}{\bf 1}_{\{s<\tau_{y_1-h}(Y)\}}\,
\partial_y^2 u_{f,h}^\varepsilon(X(s),Y(s))\,b^2(Y(s))\,{\bf 1}_{\{Y(s)<y_1-h\}}\,\dd s,
\]
since $Y_h$ is absorbing at $y_1-h$, and $b(y_1-h)=0$.
But $s\,=\,A_{A^{-1}_s}$, for all $s<\tau_{y_1-h}(Y)$, a.s.,
and hence, when $\varepsilon\downarrow 0$, the second integral converges almost surely to 
\[
-\int_0^{t_N}{\bf 1}_{\{s<\tau_{y_1-h}(Y)\}}\,
{\cal L}_x u_f(X(s),Y(s))\,{\bf 1}_{(0,y_1-h]}(Y(s))\,\dd s,
\]
applying time change and partial integration as in the proof of (a). 

By dominated convergence, when $\varepsilon\downarrow 0$, the first integral of (\ref{toFind}) converges to
\[
\int_0^{t_N}{\bf 1}_{\{s<\tau_{y_1-h}(Y)\}}\,
{\cal L}_x u_{f}(X(s),Y(s))\,\dd s
\,+
\int_0^{t_N}{\bf 1}_{\{s\ge\tau_{y_1-h}(Y)\}}\,
{\cal L}_x u_{f}(X(s),y_1-h)
\,\dd s,
\]
so that summing up gives the following almost sure limit of (\ref{toFind}),
\[
\int_0^{t_N}{\bf 1}_{\{s<\tau_{y_1-h}(Y)\}}\,
{\cal L}_x u_f(X(s),0)\,{\bf 1}_{\{0\}}(Y(s))\,\dd s
\,+
\int_0^{t_N}{\bf 1}_{\{s\ge\tau_{y_1-h}(Y)\}}\,{\cal L}_x u_f(X(s),y_1-h)\,\dd s,
\]
when $\varepsilon\downarrow 0$.

Of course, being the limit of a left-hand and a right-hand side of It\^o's formula, respectively,
(\ref{itoYh}) almost surely equals the $\varepsilon$-limit of (\ref{toFind}),
and hence
\begin{align}\label{rightHandAbsorbed}
0\,&\stackrel{a.s.}{=}
u_{f}(X(t_N),Y(t_N))\,{\bf 1}_{\{t_N<\tau_{y_1-h}(Y)\}}
	+u_f(X(t_N),y_1-h)\,{\bf 1}_{\{t_N\ge\tau_{y_1-h}(Y)\}}
	-f(x)\nonumber\\
	&\,-
\int_0^{t_N}{\bf 1}_{\{s<\tau_{y_1-h}(Y)\}}\,
\partial_y^+ u_f(X(s),0)\,\dd L_s^0(Y)\nonumber\\
	&\,-
\int_0^{t_N}{\bf 1}_{\{s<\tau_{y_1-h}(Y)\}}\,
{\cal L}_x u_f(X(s),0)\,{\bf 1}_{\{0\}}(Y(s))\,\dd s
\,-
\int_0^{t_N}{\bf 1}_{\{s\ge\tau_{y_1-h}(Y)\}}{\cal L}_x u_f(X(s),y_1-h)\,\dd s
\nonumber\\
&\,-
\sum_{i=1}^d\sum_{k=1}^p\int_0^{t_N}{\bf 1}_{\{s<\tau_{y_1-h}(Y)\}}\,
\partial_i u_{f}(X(s),Y(s))\,\sigma_{ik}(X(s))\,\dd B_k(s)\nonumber\\
&\,-
\sum_{i=1}^d\sum_{k=1}^p\int_0^{t_N}{\bf 1}_{\{s\ge\tau_{y_1-h}(Y)\}}\,
\partial_i u_f(X(s),y_1-h)\,\sigma_{ik}(X(s))\,\dd B_k(s)\nonumber\\
&\,-
\sqrt{2}\int_0^{t_N}{\bf 1}_{\{s<\tau_{y_1-h}(Y)\}}\,
\partial_y^+ u_{f}(X(s),Y(s))\,b(Y(s))\,\dd B(s).
\end{align}

Eventually, we choose a whole sequence of $h$-values converging to zero.
Since countably many $h$-values still form a set of Lebesgue measure zero,
$b(y_1-h)$ can be assumed to be zero, for any $h$ in this countable set,
and we have to show that the $h$-limit of the above equation's
right-hand side almost surely equals $\mathfrak{Z}(t_N)$.

First, 
$\{\tau_{y_1-h}(Y)>t_N\}\uparrow\{\tau_{y_1}(Y)>t_N\}$,
when $h\downarrow 0$, and $\mathfrak{Z}(t_N)$ equals 
the right-hand side of (\ref{rightHandAbsorbed}), on each $\{\tau_{y_1-h}(Y)>t_N\}$. 
Therefore, without loss of generality,
we only show that the $h$-limit of this right-hand side
almost surely equals $\mathfrak{Z}(t_N)$ under the assumption that
$\tau_{y_1}(Y)$ is finite.

Under this assumption,
$\tau_{y_1}(Y)-\tau_{y_1-h}(Y)\to 0$, almost surely, when $h\downarrow 0$,
and hence, using dominated convergence,
all summands on the right-hand side of (\ref{rightHandAbsorbed}),
\underline{except the last one}, can be shown to 
converge in $L^2(\pp_{\!\!x})$ to their $\mathfrak{Z}(t_N)$-counterparts,
when $h\downarrow 0$, in a straight forward way
(recall that $t_N$ satisfies $t_N\le\tau_N(L_{\fatdot}^0(Y))$, by definition).

Identifying the limit of the last summand is more involved
because $\partial_y^+ u_f(\cdot,y){\bf 1}_{\{\|\cdot\|\le N\}}$ 
may become unbounded, when $y$ approaches $y_1$.
Recall that $\partial_y^+ u_f(\cdot,0){\bf 1}_{\{\|\cdot\|\le N\}}$
is bounded by Corollary \ref{BCyreg}.

However, as the left-hand side of (\ref{rightHandAbsorbed}) is zero,
if all other summands converge in $L^2(\pp_{\!\!x})$,
then the last summand does, too, so that
\[
\mbox{$\lim_{h\downarrow 0}$}\;
\ee_x\left[\left(\rule{0pt}{12pt}\right.
\int_0^{t_N}{\bf 1}_{\{s<\tau_{y_1-h}(Y)\}}\,
\partial_y^+ u_{f}(X(s),Y(s))\,b(Y(s))\,\dd B(s)
\left.\rule{0pt}{12pt}\right)^2\right]
\,<\,+\infty.
\]
But,
\begin{align*}
&\mbox{$\lim_{h\downarrow 0}$}\;
\ee_x\left[\left(\rule{0pt}{12pt}\right.
\int_0^{t_N}{\bf 1}_{\{s<\tau_{y_1-h}(Y)\}}\,
\partial_y^+ u_{f}(X(s),Y(s))\,b(Y(s))\,\dd B(s)
\left.\rule{0pt}{12pt}\right)^2\right]\\
&=\,
\mbox{$\lim_{h\downarrow 0}$}\;
\ee_x\,[
\int_0^{t_N}{\bf 1}_{\{s<\tau_{y_1-h}(Y)\}}
\left(\rule{0pt}{11pt}\right.
\partial_y^+ u_{f}(X(s),Y(s))\,b(Y(s))
\left.\rule{0pt}{11pt}\right)^2
\dd s\,]\\
&=\,
\ee_x\,[
\int_0^{t}{\bf 1}_{\{s<t_N\wedge\tau_{y_1}(Y)\}}
\left(\rule{0pt}{11pt}\right.
\partial_y^+ u_{f}(X(s),Y(s))\,b(Y(s))
\left.\rule{0pt}{11pt}\right)^2
\dd s\,],
\end{align*}
where the last line follows by monotone convergence.
The above justifies that
\[
{\bf 1}_{[0,t_N\wedge\tau_{y_1}(Y))}
\left(\rule{0pt}{11pt}\right.
\partial_y^+ u_{f}(X(\cdot),Y(\cdot))\,b(Y(\cdot))
\left.\rule{0pt}{11pt}\right)^2
\in
L^2(\Omega\times[0,t]),
\]
and hence
\begin{align*}
&\mbox{$\lim_{h\downarrow 0}$}\;
\ee_x\left[\left(\rule{0pt}{12pt}\right.
\int_0^{t_N}{\bf 1}_{[\tau_{y_1-h}(Y),\tau_{y_1}(Y))}(s)\,
\partial_y^+ u_{f}(X(s),Y(s))\,b(Y(s))\,\dd B(s)
\left.\rule{0pt}{12pt}\right)^2\right]\\
&=\,
\mbox{$\lim_{h\downarrow 0}$}\;
\ee_x\,[
\int_0^{t_N}{\bf 1}_{[\tau_{y_1-h}(Y),\tau_{y_1}(Y))}(s)\,
\left(\rule{0pt}{11pt}\right.
\partial_y^+ u_{f}(X(s),Y(s))\,b(Y(s))
\left.\rule{0pt}{11pt}\right)^2
\dd s\,]\,=\,0,
\end{align*}
by dominated convergence, proving part (b) of the lemma.

(c) Choose $h\in(0,y_1)$, 
and define $t_N,u_{f,h},\,u_{f,h}^\varepsilon$ as in the proof of (b).
Since $Y$ is not absorbing at $y_1<\infty$, 
the local time $L^{y_1}_{\,\fatdot}(Y)$ does not vanish,
and hence the classical version of It\^o's lemma gives
\begin{align}\label{newIto}
&\hspace{1pt}u_{f,h}^\varepsilon(X({t_N}),Y({t_N}))-u_{f,h}^\varepsilon(x,0)
\,-
\int_0^{t_N}\partial_y u_{f,h}^\varepsilon(X(s),0)\,\dd L_s^0(Y)
+\int_0^{t_N}\partial_y u_{f,h}^\varepsilon(X(s),y_1)\,\dd L_s^{y_1}(Y)\nonumber\\
-&
\sum_{i=1}^d\sum_{k=1}^p\int_0^{t_N}\partial_i u_{f,h}^\varepsilon(X(s),Y(s))\,\sigma_{ik}(X(s))\,\dd B_k(s)
\,-\,
\sqrt{2}\int_0^{t_N}\partial_y u_{f,h}^\varepsilon(X(s),Y(s))\,b(Y(s))\,\dd B(s)\nonumber\\
&\hspace{1.2cm}\,\stackrel{a.s.}{=}
\int_0^{t_N}{\cal L}_x u_{f,h}^\varepsilon(X(s),Y(s))\,\dd s
\,+
\int_0^{t_N}\partial_y^2 u_{f,h}^\varepsilon(X(s),Y(s))\,b^2(Y(s))\,\dd s.
\end{align}

However, as $u_{f,h}^\varepsilon(x,\cdot)$ is constant on $[y_1-h/2,\infty)$, 
for all $x\in\bfR^d$, the term involving $L^{y_1}_{\,\fatdot}(Y)$ vanishes,
so that the $\varepsilon$-limit of the above left-hand side almost surely equals
\[
\begin{array}{c}
u_{f,h}(X({t_N}),Y({t_N}))-f(x)\\
\rule{0pt}{15pt}
-\int_0^{t_N}\partial^+_y u_f(X(s),0)\,\dd L_s^0(Y)\\
\rule{0pt}{15pt}
-\sum_{i=1}^d\sum_{k=1}^p\int_0^{t_N}\partial_i u_{f,h}(X(s),Y(s))\,\sigma_{ik}(X(s))\,\dd B_k(s)\\
\rule{0pt}{15pt}
-\,\sqrt{2}\int_0^{t_N}\partial^+_y u_{f,h}(X(s),Y(s))\,b(Y(s))\,\dd B(s),
\end{array}
\]
by the same arguments used in the proof of (a).

Furthermore, unlike in case (b) where $y_1$ is absorbing,
it now must hold that $\hat{m}([0,y_1])<+\infty$ (see the beginning of Remark \ref{goody}),
and hence $s\,=\,A_{A^{-1}_s}$, for all $s\ge 0$, a.s. Therefore, the $\varepsilon$-limit
of the second integral on the right-hand side of (\ref{newIto}) can almost surely be given by
\begin{equation}\label{2nd integral}
-\int_0^{t_N}\hspace{-8pt}
{\cal L}_x u_f(X(s),Y(s))\,{\bf 1}_{(0,y_1-\frac{h}{2})}(Y(s))\,\dd s\,
+\!\int_0^{A^{-1}_{t_N}}\hspace{-5pt}
\Delta[\partial_y^+ u_{f,h}(X(A_s),y_1-\frac{h}{2}\,)]\,\dd L^{y_1-\frac{h}{2}}_s(W)/2,
\end{equation}
again applying time change and partial integration as in the proof of (a). 
Here, the `artificial' jump of $\partial_y^+ u_{f,h}(X(A_s),\cdot)$ at $y_1-h/2$, 
which we created when extending $u_{f}$ to half-space, equals
\begin{equation}\label{extra jump}
\Delta[\partial_y^+ u_{f,h}(X(A_s),y_1-h/2)]
\,=\,
-\partial_y^- u_{f}(X(A_s),y_1-h/2).
\end{equation}

All in all, letting $\varepsilon$ go to zero on both sides of (\ref{newIto}) yields
\begin{align*}
0\,&\stackrel{a.s.}{=}
u_{f,h}(X(t_N),Y(t_N))\,{\bf 1}_{\{Y(t_N)<y_1\}}
	+u_f(X(t_N),y_1-\frac{h}{2}\,)\,{\bf 1}_{\{Y(t_N)=y_1\}}
	-f(x)\\
&\,-
\int_0^{t_N}\partial_y^+ u_f(X(s),0)\,\dd L_s^0(Y)
+\int_0^{A^{-1}_{t_N}}
\partial_y^- u_{f}(X(A_s),y_1-\frac{h}{2}\,)\,\dd L^{y_1-\frac{h}{2}}_s(W)\times 1/2\\
	&\,-
	\int_0^{t_N}{\cal L}_x u_f(X(s),0)\,{\bf 1}_{\{0\}}(Y(s))\,\dd s
	\,-
	\int_0^{t_N}{\cal L}_x u_f(X(s),y_1-\frac{h}{2})\,{\bf 1}_{\{Y(s)\ge y_1-\frac{h}{2}\}}\,\dd s\\
&\,-
\sum_{i=1}^d\sum_{k=1}^p\int_0^{t_N}
\partial_i u_{f,h}(X(s),Y(s))\,\sigma_{ik}(X(s))\,{\bf 1}_{\{Y(s)<y_1\}}\,\dd B_k(s)\\
&\,-
\sum_{i=1}^d\sum_{k=1}^p\int_0^{t_N}
\partial_i u_f(X(s),y_1-\frac{h}{2})\,\sigma_{ik}(X(s))\,{\bf 1}_{\{Y(s)=y_1\}}\,\dd B_k(s)\\
&\,-
\sqrt{2}\int_0^{t_N}
\partial_y^+ u_{f,h}(X(s),Y(s))\,b(Y(s))\,\dd B(s),
\end{align*}
and it needs to be shown that
the $h$-limit of the above right-hand side almost surely coincides with $\mathfrak{Z}(t_N)$,
where $\mathfrak{Z}(t)$ is the case-(c)-version of what has been defined in the proofs of (a),(b).

By Corollary \ref{BCyreg}, $\partial_y^+ u_f$ always behaves well
near the boundary of $\bfR^d\times(0,y_1)$ at zero.
In case (c), Corollary \ref{BCyreg} also implies that $\partial_y^+ u_f$
behaves well near the boundary of $\bfR^d\times(0,y_1)$ at $y_1$.
So, all terms except 
$\int_{[0,A^{-1}_{t_N}]}\partial_y^- u_{f}(X(A_s),y_1-\frac{h}{2})\,\dd L^{y_1-\frac{h}{2}}_s(W)/2$
can be shown to converge almost surely or in $L^2(\pp_{\!\!x})$
to their $\mathfrak{Z}(t_N)$-counterparts in a straight forward way, when $h\downarrow 0$.

Below, we verify that
\[
\frac{1}{2}
\int_0^{A^{-1}_{t_N}}\partial_y^- u_{f}(X(A_s),y_1-\frac{h}{2})\,\dd L^{y_1-\frac{h}{2}}_s(W)
\stackrel{a.s.}{\longrightarrow}
\int_0^{t_N}\partial_y^- u_f(X(s),y_1-)\,\dd L_s^{y_1}(Y),
\]
when $h\downarrow 0$, finishing the proof of part (c) and the lemma.

Since \cite[Lemma (6.34)(i)]{AS1998} gives
\[
\int_0^{t_N}\partial_y^- u_f(X(s),y_1-)\,\dd L_s^{y_1}(Y)
\,\stackrel{a.s.}{=}\,
\int_0^{A^{-1}_{t_N}}\partial_y^- u_f(X(A_s),y_1-)\,\dd L_s^{y_1}(W)/2\,,
\]
we can almost surely bound
\[
|\int_0^{A^{-1}_{t_N}}\partial_y^- u_{f}(X(A_s),y_1-\frac{h}{2})\,\dd L^{y_1-\frac{h}{2}}_s(W)/2
\,-
\int_0^{t_N}\partial_y^- u_f(X(s),y_1-)\,\dd L_s^{y_1}(Y)\,|
\]
by
\begin{align*}
&\hspace{1cm}\frac{1}{2}\int_0^{A^{-1}_{t_N}}|\,
\partial_y^- u_{f}(X(A_s),y_1-\frac{h}{2})
-
\partial_y^- u_{f}(X(A_s),y_1-)\,|
\,\dd L^{y_1-\frac{h}{2}}_s(W)\\
+\,&\frac{1}{2}\,|
\int_0^{A^{-1}_{t_N}}\partial_y^- u_{f}(X(A_s),y_1-)\,\dd L^{y_1-\frac{h}{2}}_s(W)
\,-
\int_0^{A^{-1}_{t_N}}\partial_y^- u_{f}(X(A_s),y_1-)\,\dd L^{y_1}_s(W)\,|,
\end{align*}
and the task is to show that both summands vanish a.s., when $h\downarrow 0$.

First, observe that $\|X(A_s)\|\le N$, for $s\le A^{-1}_{t_N}$, a.s.,
and that $0<h/2<y_1/2$ by our choice of $h$. Therefore,
the first summand is bounded by
\vspace{-15pt}
\[
\sup_{\|x\|\le N}
|\,\partial_y^- u_{f}(x,y_1-\frac{h}{2})-\partial_y^- u_{f}(x,y_1-)|
\;\;\times
\overbrace{
\sup_{0\le h'\le\frac{y_1}{2}}\,L_{A^{-1}_{t_N}}^{y_1-h'}(W),
}^{a.s.\,finite}
\]
where it follows from the proof of Corollary \ref{BCyreg} that
\[
\sup_{\|x\|\le N}
|\,\partial_y^- u_{f}(x,y_1-\frac{h}{2})-\partial_y^- u_{f}(x,y_1-)|
\longrightarrow 0,\quad h\downarrow 0.
\]
Thus the limit of the first summand almost surely vanishes.

For the second summand, note that $A^{-1}_{t_N}$ is almost surely finite,
and hence, for almost every $\omega\in\Omega$, and each $y\in\bfR$,
$s\mapsto L_s^y(W)(\omega)$ can be considered a continuous distribution function 
of a finite measure $\nu_y(\omega)$ on $[0,A^{-1}_{t_N(\omega)}(\omega)]$.
Since $L_s^y(W)(\omega)$ is continuous in $y$, too, 
the measures $\nu_{y_1-h/2}(\omega)$ converge weakly to $\nu_{y_1}(\omega)$, 
when $h$ goes to zero.
As a consequence, when $h\downarrow 0$, 
the second summand converges almost surely to zero,
because $s\mapsto A_s$ is continuous, on the one hand, and because 
$x\mapsto\partial_y^- u_{f}(x,y_1-)$
is by assumption a bounded continuous function, for $\|x\|\le N$, on the other.
\end{proof}
%
%
\begin{proof}[Proof of \textbf{Lemma \ref{analytic for Y}}]
Recall that $y_1$ is assumed to be finite.

(i) It has already been pointed out at the beginning of Remark \ref{goody}
that if $Y$ is not absorbing at $y_1$ then $\hat{m}([0,y_1])<+\infty$.
Vice versa, if $\hat{m}([0,y_1])<+\infty$,
then the diffusion $Y$ constructed in the proof of Lemma A.1 in the Appendix
can never be absorbing at $y_1$, because $A_t<+\infty$, a.s., for all $t\ge 0$.
Thus, $Y$ is not absorbing at $y_1$ if and only if $\hat{m}([0,y_1])<+\infty$,
which is equivalent to the statement to be proven under (i).

(ii) Again by the construction given in the proof of Lemma A.1 in the Appendix,
$\tau_{y_1}(Y)\stackrel{a.s.}{=}+\infty$
if and only if $\lim_{t\uparrow\tau_{y_1}(W)}A_t\stackrel{a.s.}{=}+\infty$,
and it is a special case of \cite[Proposition (A1.8)]{AS1998} that
$\lim_{t\uparrow\tau_{y_1}(W)}A_t\stackrel{a.s.}{=}+\infty$
if and only if
$\int_{[0,y_1)}(y_1-y)\,m(\dd y)=+\infty$.
So, since $\int_{[0,y_1)}(y_1-y)\,m(\dd y)=+\infty$ if and only if
$\int_{[0,y_1)}(y_1-y)\,\hat{m}(\dd y)=+\infty$,
part (ii) of the lemma follows.
\end{proof}

\begin{proof}[Proof of \textbf{Theorem \ref{trace_gen_theorem}}]
(a) Fix $x\in\bfR^d$, and choose an arbitrary but small $t>0$.
Define the stopping times $\tau_N,\,S_N$ as at the beginning of 
the proof of Lemma \ref{ito}, and set 
\begin{equation}\label{defi t_N}
t_N
\,=\,
T_t\wedge\tau_N\wedge S'_N,
\quad\mbox{where}\quad
S'_N\,=\,S_N\wedge\tau_N(Y).
\end{equation}

Then, for fixed $N\ge 1$, Lemma \ref{ito}(a) yields 
	\begin{align*}
		u_f(X(t_N),Y(t_N))-f(x)  
		&\,\stackrel{a.s.}{=}
		\int_0^{t_N}\partial_y^+ u_f(X(s),0)\,\dd L_s^0(Y) 
		\,+
		\int_0^{t_N}{\cal L}_x u_f(X(s),0)\,{\bf 1}_{\{0\}}(Y(s))\,\dd s\\
		&\hspace*{-4.3cm}+
		\sum_{i=1}^d\sum_{k=1}^p\int_0^{t_N}\partial_i u_f(X(s),Y(s))\,\sigma_{ik}(X(s))\,\dd B_k(s)
		\,+
		\sqrt{2}\int_0^{t_N}\partial_y^+ u_f(X(s),Y(s))\,b(Y(s))\,\dd B(s);
	\end{align*}
and since $t_N\le\tau_N\wedge S'_N$,
all stochastic integrals above  have zero expectation, so that
\[
\lim_{N\uparrow\infty}\,
\ee_x[u_f(X(t_N),Y(t_N))-f(x)]
\,=\,
\lim_{N\uparrow\infty}\,
\ee_x\left[
\int_0^{t_N}\left(\rule{0pt}{12pt}
\partial_y^+ u_f(X(s),0)+m_0\,{\cal L}_x u_f(X(s),0)
\right)\dd L_s^0(Y)
\right]\!,
\]
where we also used 
\begin{equation}\label{ds2dL}
\int_0^{t_N}{\cal L}_x u_f(X(s),0)\,{\bf 1}_{\{0\}}(Y(s))\,\dd s
\,\stackrel{a.s.}{=}\,
\int_0^{t_N}m_0\,{\cal L}_x u_f(X(s),0)\,\dd L_s^0(Y),
\end{equation}
which is an easy consequence of \cite[Theorem (5.27)]{AS1998},
when the scale function is the identity.

Now, recall that $Y(t)=W(A_t^{-1}),\,t\ge 0$, where $A_t$ is given in the 
proof of Lemma A.1 in the Appendix. 
Then, since \cite[Lemma (6.34)(i)]{AS1998} yields 
\[
L^0_t(Y)\,=\,\frac{1}{2}L^0_{A_t^{-1}}(W),\quad t\ge 0,\;\mbox{a.s.},
\]
we have that
\begin{equation}\label{forAbsorb}
L^0_\infty(Y)\,\stackrel{a.s.}{=}\,+\infty
\quad\mbox{if and only if}\quad
A_\infty^{-1}\,\stackrel{a.s.}{=}\,+\infty.
\end{equation}
But, 
$A_\infty^{-1}\,\stackrel{a.s.}{=}\,+\infty$
follows by the same arguments used to show that
$s\,=\,A_{A^{-1}_s}$, for all $s\ge 0$, a.s.,
in the proof of part (a) of Lemma \ref{ito},
so we do have that
$L^0_\infty(Y)\,\stackrel{a.s.}{=}\,+\infty$,
for part (a) of this proof.

As a consequence,
since $\tau_N\wedge S'_N$ grows to infinity,
$N\to\infty$, we have that,
for almost every $\omega\in\Omega$, there exists $N(\omega)$ such that
$t_N(\omega)=T_t(\omega)$, 
for all $N\ge N(\omega)$,
and hence
\[
\lim_{N\uparrow\infty}\,
\ee_x[\,u_f(X(t_N),Y(t_N))-f(x)\,]
\,=\,
\ee_x[\,f(X(T_t))-f(x)\,],
\]
by dominated convergence, even if $f$ was not continuous.
	
	On the other hand, since $t_N \leq T_t,\,N \geq 1$, 
and since $\partial^+_y u_f(\cdot, 0)+m_0\,\mathcal{L}_xu_f(\cdot, 0)$ is bounded,
again by dominated convergence,
	\begin{align*}
		\lim_{N \uparrow \infty} \ee_x\Big[&\int_{0}^{t_N} \Big(\partial^+_y u_f(X(s), 0) + m_0\,\mathcal{L}_x u_f(X(s), 0)  \Big) \dd L^0_s(Y)  \Big] \\
		&= \ee_x\Big[\int_{0}^{T_t} \Big(\partial^+_y u_f(X(s), 0) + m_0\,\mathcal{L}_x u_f(X(s), 0)  \Big) \dd L^0_s(Y)  \Big] \\
		&= \ee_x\Big[\int_{0}^{t} \Big(\partial^+_y u_f(X(T_s), 0) + m_0\,\mathcal{L}_x u_f(X(T_s), 0)  \Big) \dd s  \Big].
	\end{align*}

All in all, to finish the proof of part (a) of the theorem, it remains to show that
\begin{align*}
\mbox{$\lim_{t\downarrow 0}$}\,
&
\ee_x\Big[\,
\frac{1}{t}
\int_0^{t}\left(\rule{0pt}{12pt}
\partial_y^+ u_f(X(T_s),0)+m_0\,{\cal L}_x u_f(X(T_s),0)
\right)\dd s
\,\Big]\\
=&\,
\partial_y^+ u_f(x,0)+m_0\,{\cal L}_x u_f(x,0),
\end{align*}
which follows by dominated convergence, because  
$\partial^+_y u_f(\cdot, 0)+m_0\,\mathcal{L}_xu_f(\cdot, 0)\in C_b(\bfR^d)$.

(b) The proofs in the cases (b1) and (b2) are identical for a large part,
and that's why we name this large part of the proof (b),
and we only go into the differences between (b1) \&\ (b2) at the end.
By Lemma \ref{analytic for Y}, the point $y_1$ will be absorbing in this case.

If $y_1<\infty$ and $Y$ is absorbing at $y_1$,
then $A_\infty^{-1}=\tau_{y_1}(W)$,
by the construction of $Y$---see the Appendix, proof of Lemma A.1.
Of course, $\tau_{y_1}(W)<+\infty$, a.s., and hence 
$L^0_\infty(Y)\,<\,+\infty$, a.s., too, by (\ref{forAbsorb}).
As a consequence, $\pp_{\!\!x}(\{T_t=+\infty\})>0$, for any $t>0$,
and this positive probability can be determined using Lemma A.1.
Indeed, since 
$\ee_x[e^{-\lambda T_t}{\bf 1}_{\{T_t<+\infty\}}]=e^{-t\psi(\lambda)},\,\lambda>0$,
we obtain that
\[
\pp_{\!\!x}(\{T_t=+\infty\})
\,=\,
1-e^{-t\psi(0)},
\]
when $\lambda$ goes to zero.

Note that the above reasoning also implies that $\psi(0)$ has to be \underline{positive} in case (b),
though we obviously had $\psi(0)=0$ in case (a).

Now, fix $x\in\bfR^d,\,t>0$, and define $t_N$ by (\ref{defi t_N}), 
but using $S'_N=S_N\wedge\tau_{y_1-1/N}(Y_h)$, instead,
where $Y_h$ again denotes the process $[Y(t\wedge\tau_{y_1-h}(Y)),\,t\ge 0]$,
for some $h\in(0,y_1)$.
Of course, if $Y$ does hit $y_1$ then it would hit it after 
hitting $y_1-h$, that is
\begin{equation}\label{zero before}
\pp_{\!\!x}(\{t_N<\tau_{y_1}(Y)\})\,=\,1,
\quad
N=1,2,\dots,
\end{equation}
though $\tau_{y_1-h}(Y)$ converges to $\tau_{y_1}(Y)$, when $h\downarrow 0$,
whether $\tau_{y_1}(Y)$ is finite or not.
Furthermore, if $y_1<\infty$ is absorbing then
$L^0_{\tau_{y_1}(Y)}(Y)$ always equals $L^0_\infty(Y)$ almost surely,
because $Y$ can only be absorbed at $y_1>0$ after hitting zero for the last time.

All in all, for fixed $N\ge 1$, Lemma \ref{ito}(b) yields
\begin{align*}
&\,
                u_f(X(t_N),Y(t_N))-f(x)  \\
                &\,\stackrel{a.s.}{=}
                \int_0^{t_N}\partial_y^+ u_f(X(s),0)\,\dd L_s^0(Y) 
                \,+
                \int_0^{t_N}{\cal L}_x u_f(X(s),0)\,{\bf 1}_{\{0\}}(Y(s))\,\dd s
                \,+\,\mbox{mart}(t_N),
\end{align*}
where $\mbox{mart}(t_N)$ denotes the sum of the stochastic integrals, at stopping time $t_N$.

Taking expectations on both sides, we therefore obtain that
\[
\lim_{N\uparrow\infty}\,
\ee_x[\,u_f(X(t_N),Y(t_N))-f(x)\,]
\,=\,
\lim_{N\uparrow\infty}\,
\ee_x\left[
\int_0^{t_N}\left(\rule{0pt}{12pt}
\partial_y^+ u_f(X(s),0)+m_0\,{\cal L}_x u_f(X(s),0)
\right)\dd L_s^0(Y)
\right]\!,
\]
using (\ref{ds2dL}) as in the proof of part (a).

First, we deal with the limit of the above right-hand side.
Time change yields
\begin{align*}
&\int_0^{t_N}\left(\rule{0pt}{12pt}
\partial_y^+ u_f(X(s),0)+m_0\,{\cal L}_x u_f(X(s),0)
\right)\dd L_s^0(Y)\\
&=
\int_0^{L^0_{t_N}(Y)}\left(\rule{0pt}{12pt}
\partial_y^+ u_f(X(T_s),0)+m_0\,{\cal L}_x u_f(X(T_s),0)
\right)\dd s,
\end{align*}
where $t_N$ grows to $T_t\wedge\tau_{y_1-h}(Y)$, when $N\to\infty$.
So, since  $\partial^+_y u_f(\cdot, 0)+m_0\,\mathcal{L}_xu_f(\cdot, 0)$ is bounded,
by dominated convergence,
the limit of the right-hand side equals
\begin{align*}
&\ee_x\Big[
\int_{0}^{t\wedge L^0_{T_t\wedge\tau_{y_1-h}(Y)}(Y)}
\Big(\partial^+_y u_f(X(T_s), 0) + m_0\,\mathcal{L}_x u_f(X(T_s), 0)  \Big) 
\dd s\,
{\bf 1}_{\{L^0_{\tau_{y_1-h}(Y)}(Y)<L^0_\infty(Y)\}} \Big] \\
+\;&
\ee_x\Big[
\int_{0}^{t\wedge L^0_\infty(Y)} 
\Big(\partial^+_y u_f(X(T_s), 0) + m_0\,\mathcal{L}_x u_f(X(T_s), 0)  \Big) 
\dd s\,
{\bf 1}_{\{L^0_{\tau_{y_1-h}(Y)}(Y)\ge L^0_\infty(Y)\}}  \Big],
\end{align*}
which converges to 
\[
\ee_x\Big[
\int_{0}^{t\wedge L^0_\infty(Y)} 
\Big(\partial^+_y u_f(X(T_s), 0) + m_0\,\mathcal{L}_x u_f(X(T_s), 0)  \Big) 
\dd s \Big],
\]
when $h\downarrow 0$.

Next, it is easy to see that,
for almost every $\omega\in\{T_t<\tau_{y_1-h}(Y)\}$,
there exists $N(\omega)$ such that
$t_N(\omega)\,=\,T_t(\omega)$, for all $N\ge N(\omega)$.
Also, since $\hat{m}([0,y_1-h])<+\infty$, we know that
$\tau_{y_1-h}(Y)$ is almost surely finite,
and therefore 
$\{T_t\ge\tau_{y_1-h}(Y)\}\stackrel{a.s.}{=}\{T_t>\tau_{y_1-h}(Y)\}$,
because the process $Y$ cannot be at zero and $y_1-h$ at the same time.
As a consequence,
for almost every $\omega\in\{T_t\ge\tau_{y_1-h}(Y)\}$,
there exists $N(\omega)$ such that
$t_N(\omega)\,=\,\tau_{y_1-h}(Y)(\omega)$, for all $N\ge N(\omega)$,
and we obtain that
\begin{align*}
\lim_{N\uparrow\infty}\,
&\ee_x[\,u_f(X(t_N),Y(t_N))-f(x)\,]\\
\,=\,&
\ee_x[\Big(f(X(T_t))-f(x)\Big){\bf 1}_{\{T_t<\tau_{y_1-h}(Y)\}}]\\
\,+\,&
\ee_x[\Big(u_f(X(\tau_{y_1-h}(Y)),y_1-h)-f(x)\Big){\bf 1}_{\{T_t\ge\tau_{y_1-h}(Y)\}}],
\end{align*}
by dominated convergence, only using boundedness of $u_f$, but not continuity.

Recall that 
$L^0_{\tau_{y_1}(Y)}(Y)\stackrel{a.s.}{=}L^0_\infty(Y)$,
which implies $\pp_{\!\!x}(\{T_t<\infty\}\setminus\{T_t<\tau_{y_1}(Y)\})=0$,
and hence
\begin{align*}
\lim_{h\downarrow 0}\,
\ee_x[\Big(f(X(T_t))-f(x)\Big){\bf 1}_{\{T_t<\tau_{y_1-h}(Y)\}}]
\,=\,&\lim_{h\downarrow 0}\,
\ee_x[\Big(f(X(T_t))-f(x)\Big)
{\bf 1}_{\{T_t<\tau_{y_1-h}(Y)\}\cap\{T_t<\infty\}}] \\
\,=\,&\,
\ee_x[\Big(f(X(T_t))-f(x)\Big){\bf 1}_{\{T_t<\infty\}}]
\end{align*}
as well as
\[
\lim_{h\downarrow 0}\,
\ee_x[\,f(x)\,{\bf 1}_{\{T_t\ge\tau_{y_1-h}(Y)\}}]
\,=\,
\ee_x[\,f(x)\,{\bf 1}_{\{T_t=\infty\}}],
\]
both by dominated convergence.

Eventually, when treating the remaining limit
\begin{equation}\label{toVanish}
\lim_{h\downarrow 0}\,
\ee_x[\,u_f(X(\tau_{y_1-h}(Y)),y_1-h)\,{\bf 1}_{\{T_t\ge\tau_{y_1-h}(Y)\}}],
\end{equation}
we have to differ between the two cases (b1) and (b2),
where $\tau_{y_1}(Y)<+\infty$, a.s., in case (b1), by Lemma \ref{analytic for Y}.

By dominated convergence, we only have to discuss
$\lim_{h\downarrow 0}\,|u_f(X(\tau_{y_1-h}(Y)),y_1-h)|$,
because if this limit vanishes almost surely then so does (\ref{toVanish}).
However, $\lim_{h\downarrow 0}\,|u_f(X(\tau_{y_1-h}(Y)),y_1-h)|$
trivially vanishes, when applying the assumptions made in either (b1) or (b2).

On the whole, we have justified that
\[
\ee_x[\,f(X(T_t)){\bf 1}_{\{T_t<\infty\}}-f(x)\,]
\,=\,
\ee_x\Big[
\int_{0}^{t\wedge L^0_\infty(Y)} 
\Big(\partial^+_y u_f(X(T_s), 0) + m_0\,\mathcal{L}_x u_f(X(T_s), 0)  \Big) 
\dd s \Big],
\]
where
\begin{align*}
\mbox{$\lim_{t\downarrow 0}$}\,
&
\ee_x\Big[\,
\frac{1}{t}
\int_0^{t\wedge L^0_\infty(Y)}\left(\rule{0pt}{12pt}
\partial_y^+ u_f(X(T_s),0)+m_0\,{\cal L}_x u_f(X(T_s),0)
\right)\dd s
\,\Big]\\
=&\,
\partial_y^+ u_f(x,0)+m_0\,{\cal L}_x u_f(x,0),
\end{align*}
by the same arguments as in the proof of part (a),
only taking into account that, when $t\downarrow 0$,
for almost every $\omega\in\Omega$, it will eventually happen that
$t<L^0_\infty(Y)(\omega)$.

Since $f(X(T_t)){\bf 1}_{\{T_t=\infty\}}=f(\dagger){\bf 1}_{\{T_t=\infty\}}=0$,
the proof is complete in both cases (b1) and (b2).

(c) By Lemma \ref{analytic for Y}, the point $y_1<\infty$ is not absorbing,
and $\hat{m}([0,y_1])<+\infty$ implies both
$L^0_\infty(Y)\,\stackrel{a.s.}{=}\,+\infty$
as well as $\psi(0)=0$, as in the proof of part (a).

Now, fix $x\in\bfR^d,\,t>0$, and define $t_N$ by (\ref{defi t_N}), 
but using $S'_N=S_N$, instead---there is no localisation needed w.r.t.\ $Y$
because of Corollary \ref{BCyreg}.
Since $\partial_y^- u_f(\cdot,y_1-)$ is continuous,
Lemma \ref{ito}(c) yields
\begin{align*}
&
u_f(X(t_N),Y(t_N))\,{\bf 1}_{\{Y(t_N)<y_1\}}
+u_f(X(t_N),y_1-)\,{\bf 1}_{\{y_1\}}(Y(t_N))-f(x)  \\
&\stackrel{a.s.}{=}
\int_0^{t_N}\partial_y^+ u_f(X(s),0)\,\dd L_s^0(Y) 
\,+
\int_0^{t_N}{\cal L}_x u_f(X(s),0)\,{\bf 1}_{\{0\}}(Y(s))\,\dd s\\
&\;
-\int_0^{t_N}\partial_y^- u_f(X(s),y_1-)\,\dd L_s^{y_1}(Y) 
\,+
\int_0^{t_N}{\cal L}_x u_f(X(s),y_1-)\,{\bf 1}_{\{y_1\}}(Y(s))\,\dd s
\,+\,\mbox{mart}(t_N),
\end{align*}
again using $\mbox{mart}(t_N)$ for the sum of the stochastic integrals, at stopping time $t_N$.

As in part (a) of the proof,
for almost every $\omega\in\Omega$, there exists $N(\omega)$ such that
$t_N(\omega)=T_t(\omega)$, 
for all $N\ge N(\omega)$, and hence
\[
\lim_{N\uparrow\infty}
u_f(X(t_N),y_1-)\,{\bf 1}_{\{y_1\}}(Y(t_N))
\,\stackrel{a.s.}{=}\,
0,
\]
because $Y$ cannot be at zero and $y_1$ at the same time.
Furthermore, 
\begin{align*}
&\int_0^{t_N}\partial_y^- u_f(X(s),y_1-)\,\dd L_s^{y_1}(Y) 
\,-
\int_0^{t_N}{\cal L}_x u_f(X(s),y_1-)\,{\bf 1}_{\{y_1\}}(Y(s))\,\dd s\\
&\stackrel{a.s.}{=}
\int_0^{t_N}\left(\rule{0pt}{12pt}
\partial_y^+ u_f(X(s),y_1-)-m_1\,{\cal L}_x u_f(X(s),y_1-)
\right)\dd L_s^{y_1}(Y),
\end{align*}
again by \cite[Theorem (5.27)]{AS1998},
and the last integral vanishes by assumption, in case (c).

All in all, we obtain that
\[
\lim_{N\uparrow\infty}\,
\ee_x[u_f(X(t_N),Y(t_N))-f(x)]
\,=\,
\lim_{N\uparrow\infty}\,
\ee_x\left[
\int_0^{t_N}\left(\rule{0pt}{12pt}
\partial_y^+ u_f(X(s),0)+m_0\,{\cal L}_x u_f(X(s),0)
\right)\dd L_s^0(Y)
\right]\!,
\]
and the rest is identical to the proof of (a).	
\end{proof}

\begin{proof}[Proof of \textbf{Corollary \ref{fundamental solution}}]
Fix $(x,y)\in\bfR^d\times(0,y_1)$, and write $X^x,Y^y,u(x,y)$
instead of $X,Y,f(x)$ whenever the notation $X,Y,f(x)$
is used in the formulation of Lemma \ref{ito}.
It is readily checked that our proof of Lemma \ref{ito}
works for these starting values, as well.

Next, recall the stopping times $\tau_N,S_N$ introduced at the beginning
of the proof of Lemma \ref{ito}, and define
\[
t_N\,=\left\{\begin{array}{ccc}
\tau_{\frac{1}{N}}(Y^y)\wedge\tau_N\wedge S_N\wedge\tau_{y_1-\frac{1}{N}}(Y^y)
&:&\mbox{$y_1<\infty$ absorbing},\\
\rule{0pt}{12pt}
\tau_{\frac{1}{N}}(Y^y)\wedge\tau_N\wedge S_N\wedge\tau_N(Y^y)
&:&\mbox{else}.
\end{array}\right.
\]
By Lemma \ref{analytic for Y}, the phrase
`$y_1<\infty$ absorbing' refers to the cases (b1) \&\ (b2) 
in Theorem \ref{trace_gen_theorem}, with (a) \&\ (c) being the remaining cases.
We will always choose $N$ big enough ensuring $1/N<y<y_1-1/N$.

The idea of proof, in all four cases, consists of the following steps:
apply Lemma \ref{ito} to $u_f(X^x(t_N),$ $Y^y(t_N))$,
take expectations on both sides,
and eventually take the limit $N\to\infty$.

Of course, taking expectations removes all stochastic integrals,
by our choice of $t_N$. Furthermore, all integrals involving
$\partial_y^+ u_f(X^x(s),0)$ or ${\cal L}_x u_f(X^x(s),0)$
disappear, because $Y^y(s)>0$, for all $s\le t_N$.
In both cases (b1) \&\ (b2), all terms involving `$y_1-$' disappear,
because $s<\tau_{y_1}(Y^y)$, for all $s\le t_N$.
In case (c), the integrals involving
$\partial_y^- u_f(X^x(s),y_1-)$ or ${\cal L}_x u_f(X^x(s),y_1-)$
disappear as in the proof of Theorem \ref{trace_gen_theorem},
using the assumed boundary condition.

All in all, we arrive at 
\begin{equation}\label{after exp}
u(x,y)\,=\left\{\begin{array}{ccl}
\ee[\,u_f(X^x(t_N),Y^y(t_N))\,]&:&\mbox{(a),(b1),(b2)},\\
\rule{0pt}{20pt}
\ee\left[\begin{array}{c}
u_f(X^x(t_N),Y^y(t_N)){\bf 1}_{\{Y^y(t_N)<y_1\}}\\
+u_f(X^x(t_N),y_1-){\bf 1}_{\{y_1\}}(Y^y(t_N))
\end{array}\right]
&:&\mbox{(c)},
\end{array}\right.
\end{equation}
before taking the limit $N\to\infty$.

Let us first consider the cases (a) \&\ (c).

Here, it follows  from the construction of $Y$ given in \cite[Chapter VI]{AS1998}
that both
\[
\pp(\{\tau_0(Y^y)<\infty\})\,=\,1
\quad\mbox{and}\quad
\tau_{\frac{1}{N}}(Y^y)\uparrow\tau_0(Y^y),\,N\to\infty,\,\mbox{a.s.}
\]
Thus, since $\tau_N\wedge S_N\wedge\tau_N(Y)$ grows to infinity, $N\to\infty$,
we also have that, for almost every $\omega\in\Omega$, there exists $N(\omega)$ 
such that $t_N(\omega)=\tau_{\frac{1}{N}}(Y^y(\omega))$, for all $N\ge N(\omega)$.

So, when taking the limit $N\to\infty$ on the right-hand side of (\ref{after exp}),
at first, one can interchange limit and expectation, since $u_f$ is bounded.
Then, in case (c), the limit of the term involving `$y_1-$' must vanish,
so that in both cases (a) \&\ (c) the limit will be
$\ee[\,f(X^x(\tau_0(Y^y)))\,]$,
because of the continuity of $u_f$ at the boundary $\bfR^d\times\{0\}$.

It remains to prove the corollary in case of (b1) \& (b2),
and it obviously suffices to show that
\begin{equation}\label{limit-b1-b2}
\renewcommand{\arraystretch}{.5}
u_f(X^x(t_N),Y^y(t_N))
\begin{array}{c}
\scriptstyle a.s.\\
\longrightarrow\\
\scriptstyle N\to\infty\end{array}
f(X^x(\tau_0(Y^y))){\bf 1}_{\{\tau_0(Y^y)<\tau_{y_1}(Y^y)\}}.
\renewcommand{\arraystretch}{1.}
\end{equation}

Here, the construction of $Y$ given in \cite[Chapter VI]{AS1998} implies
almost-sure-convergence of
\[
\tau_{\frac{1}{N}}(Y^y)\uparrow\tau_0(Y^y),\,\mbox{on $\{\tau_0(Y^y)<\tau_{y_1}(Y^y)\}$},
\quad\mbox{as well as}\quad
\tau_{y_1-\frac{1}{N}}(Y^y)\uparrow\tau_{y_1}(Y^y),
\]
when $N\to\infty$, so that (\ref{limit-b1-b2}) follows from
the behaviour of $u_f$ at the boundary $\bfR^d\times\{0,y_1\}$.
\end{proof}
\section{Discussion of our Conditions}
In this section we mainly verify the conditions of Theorem \ref{trace_gen_theorem}
in the case considered by Kwa\'snicki \&\ Mucha in \cite{KM2018},
i.e.\ ${\cal L}_x=\Delta_x$,
showing that applying our stochastic method covers a significant part
of the most general result
on the extension technique for functions of the Laplace operator.
We will nevertheless comment on the situation 
where ${\cal L}_x$ is a PDE operator more general than $\Delta_x$, too.
However, a detailed  analysis of solutions to the corresponding PDE (\ref{PDE})
is rather involved and would go beyond the scope of this paper.

First, Theorem \ref{trace_gen_theorem} is based on the existence
of a solution to (\ref{PDE}),
which is an elliptic PDE with coefficients determined by ${\cal L}_x$ and $\hat{m}$.

Here, the operator ${\cal L}_x$ is determined by the coefficients
of the SDE (\ref{SDE}), and what we need in the proof of Lemma \ref{ito},
among other things, is that
(\ref{SDE}) has a global solution, and that (\ref{PDE}) has a solution $u$
such that $u(\cdot,y)\in C^2(\bfR^d)$, for all $y\in(0,y_1)$.
So, as a minimum-condition, we require the coefficients of (\ref{SDE})
to be continuous knowing that for the wanted level of smoothness
of $u$ in the $x$-variable one would probably need smoother 
coefficients, even. We did not specify any further conditions
ensuring existence of a global solution to (\ref{SDE}),
leaving the reader with choice. The standard condition would be
linear growth, i.e.\ 
\begin{equation}\label{linearGrowth}
\|a(x)\|\,+\,\|\sigma(x)\|\,\le\,const\,(1+\|x\|),
\quad x\in\bfR^d.
\end{equation}
But, any SDE of type (\ref{SDE}) with continuous coefficients
and global solutions could be chosen for our setup,
as long as the resulting solution to (\ref{PDE}) is good enough, too.

The remaining coefficient of the PDE (\ref{PDE}) is the restriction of 
the string $\hat{m}$ to $(0,y_1)$.
Excluding the two trivial strings does not put a constraint on the general case of possible strings
as one could easily cover these two cases by other means.
Forcing the measure associated with the string to satisfy (S4), 
though, excludes a class of non-trivial strings.
But, our proofs rely upon It\^o's formula, which means that $Y$ should possess
a semimartingale decomposition, and we achieve this decomposition by
requiring $Y$ to satisfy a stochastic differential equation.
Since (S4) is a necessary condition for $Y$ to have this structure,
see \cite[Theorem (7.9)]{AS1998},
the excluded class of Krein strings cannot be covered by our technique.

Second, in order to be able to apply Theorem \ref{trace_gen_theorem},
for given ${\cal L}_x$ 
and non-trivial Krein string $\hat{m}$ satisfying (S4),
one would have to understand both regularity,
and behaviour near the boundary of $\bfR^d\times(0,y_1)$,
of solutions to (\ref{PDE}).
One way of looking into this problem would be via fundamental solutions,
and Corollary \ref{fundamental solution} suggests the following approach.

Suppose the diffusion $X$ has got absolutely continuous
transition probabilities $p_t(x,x')$, i.e.
\[
\ee[\,f(X^x(t))\,]\,=\int_{\bfR^d}p_t(x,x')f(x')\,\dd x',
\quad x\in\bfR^d,\,t\ge 0,
\]
where $f:\bfR^d\to\bfR$ can be any bounded measurable function.
Then, Corollary \ref{fundamental solution} would yield
\[
\frac{\dd\pi(x,y,\dd x')}{\dd x'}\,=
\int_{[0,\infty)}p_t(x,x')\,\gamma_y(\dd t),
\quad\mbox{in case of (a),(b2),(c)},
\]
and
\[
\frac{\dd\pi(x,y,\dd x')}{\dd x'}\,=
\int_{(0,\infty]}\int_{[0,t')}p_t(x,x')\,\gamma_y(\dd t,\dd t'),
\quad\mbox{in case of (b1)},
\]
where $\gamma_y(\dd t)$ is the law of $\tau_0(Y^y)$,
and $\gamma_y(\dd t,\dd t')$ is the joint law of $(\tau_0(Y^y),\tau_{y_1}(Y^y))$, 
respectively. Here, the laws $\gamma_y(\dd t)$ and $\gamma_y(\dd t,\dd t')$
are fully determined  by $\hat{m}$ and do not depend on $p_t(x,x')$.
Therefore, if there are special transition probabilities $p_t^\star(x,x')$
such that the solutions to (\ref{PDE}) determined by the corresponding
PDE operator satisfy the assumptions of Theorem \ref{trace_gen_theorem},
for all non-trivial Krein strings satisfying (S4), 
then the same would hold true
for any PDE operator ${\cal L}_x$ whose associated
transition probabilities $p_t(x,x')$ `behave' like $p_t^\star(x,x')$.

In the remaining part of this section, the special PDE operator will be 
the Laplace operator $\Delta_x$, and hence the associated special
transition probabilities $p_t^\star(x,x')$ are given by the heat kernel.
The above principle simply tells that our Theorem \ref{trace_gen_theorem}
should at least hold true in those cases where 
$p_t(x,x')$ associated with ${\cal L}_x$
is sufficiently close to the heat kernel, in some sense.

In the case of ${\cal L}_x=\Delta_x$ it is rather easy to verify the conditions
of Theorem \ref{trace_gen_theorem}. Indeed,
taking the Fourier transform
\[
\hat{u}(\xi,y)
\,=
\int_{\bfR^d}e^{-{\rm i}\,\xi\cdot x}\,u(x,y)\,\dd x
\]
turns the PDE (\ref{PDE}) into a family of linear 2nd order ODEs,
\[
\partial^2_y\,\hat{u}(\xi,\cdot)
\,=\,
\|\xi\|^2\,\hat{u}(\xi,\cdot)\times\hat{m}(\dd y),
\quad\mbox{in}\quad(0,y_1),
\]
which can be solved separately, for each $\xi\in\bfR^d$.
The theory of solutions to such ODEs is closely linked
to Krein's spectral theory, and this link helped
Kwa\'snicki \&\ Mucha in \cite{KM2018} to establish their
extension technique for complete Bernstein functions of the Laplacian $\Delta_x$.
In what follows, we compare their results with ours.

We at first introduce the solutions to the PDE considered in \cite{KM2018}.

Let $\hat{m}$ be a non-trivial Krein string
identified with a locally finite measure on $([0,R),{\cal B}([0,R)))$,
as discussed at the beginning of Section \ref{rr}.
Denote by $f_N(\lambda,y)$ and $f_D(\lambda,y)$ the solutions
to the integral equations
\[
f_N(\lambda,y)
\,=\,
1+\lambda\int_0^y\dd y'\int_{[0,y']}f_N(\lambda,z)\hat{m}(\dd z),
\quad y\in[0,R),
\]
and
\[
f_D(\lambda,y)
\,=\,
y+\lambda\int_0^y\dd y'\int_{[0,y']}f_D(\lambda,z)\hat{m}(\dd z),
\quad y\in[0,R),
\]
respectively, for any fixed $\lambda>0$, and define
\begin{equation}\label{ratioLim}
\psi(\lambda)
\,=\,
\lim_{y\uparrow R}\,\frac{f_N(\lambda,y)}{f_D(\lambda,y)},
\quad\lambda>0.
\end{equation}
Then, for any $f\in L^2(\bfR^d)$,
\begin{equation}\label{KMsolution}
\hat{u}(\xi,y)
\,=\,
\phi(\|\xi\|^2,y)\,\hat{f}(\xi),
\quad(\xi,y)\in\bfR^d\times(0,R),
\end{equation}
where
\[
\phi(\lambda,y)
\,=\,
f_N(\lambda,y)-\psi(\lambda)\,f_D(\lambda,y),\;\lambda>0,
\quad\mbox{and}\quad
\phi(0,y)=1-\psi(0)y,
\]
gives the Fourier transform of a weak solution to
\[
\Delta_x u\times\hat{m}(\dd y)+\partial_y^2 u\,=\,0,
\quad\mbox{in}\quad
\bfR^d\times(0,R),
\]
in the sense of distributions.
\begin{remark}\rm
(a) This statement on weak solutions 
constitutes the first part of \cite[Theorem A.2]{KM2018}.
Note that there is a typo in their definition of $\phi(\lambda,y)$
given in \cite[Subsection A.3]{KM2018},
which is `$\phi_\lambda(s)$' in their notation: 
they wrote $(\psi(\lambda))^{-1}$
instead of $\psi(\lambda)$. Another typo affects the expression
for $\phi(0,y)$ given in \cite[Subsection A.3]{KM2018}: they wrote $\psi(\lambda)$
instead of $\psi(0)$. Of course, $\psi(0)$ is short for 
$\lim_{\lambda\downarrow 0}\psi(\lambda)$.

(b) By Krein's correspondence, see \cite{K1952,KW1982},
any $\psi$ defined above is a complete Bernstein function,
that is a function which admits a representation as given in Lemma A.1
with respect  to an arbitrary Borel measure $\mu$ on $[0,\infty)$
satisfying $\int_{[0,\infty)}\mu(\dd\lambda')/(1+\lambda')<+\infty$.
Vice versa, any complete Bernstein function, except $\psi\equiv 0$,
can be given via (\ref{ratioLim}), using a non-trivial Krein string.

(c) We are aware of other definitions of complete Bernstein functions,
but they are of course equivalent to ours.
\end{remark}

Next, in \cite{KM2018}, the operators $\psi(-\Delta_x)$
are defined via Fourier transform, i.e.
\[
\mbox{\Large $\widehat{\scriptstyle[\psi(-\Delta_x)f]}$}(\xi)
\,=\,
\psi(\|\xi\|^2)\,\hat{f}(\xi),
\quad \xi\in\bfR^d.
\]
Both $\Delta_x$ and $\psi(-\Delta_x)$ are considered operators
on $L^2(\bfR^d)$, with respect to the Lebesgue measure,
and the domain of $\psi(-\Delta_x)$ consists of all 
$f\in L^2(\bfR^d)$ such that $\psi(|\cdot|^2)\,\hat{f}\in L^2(\bfR^d)$.

It is then stated in \cite[Theorem A.2]{KM2018} that,
for any complete Bernstein function $\psi\not\equiv 0$,
and any $f$ in the $L^2(\bfR^d)$-domain of $\psi(-\Delta_x)$,
the solution $u$ given by (\ref{KMsolution}) satisfies
\begin{equation}\label{theirD2N}
-\psi(-\Delta_x)f
\,=\,
\lim_{y\downarrow 0}\,\partial_y u(\cdot, y),
\quad\mbox{in}\quad L^2(\bfR^d).
\end{equation}
\begin{remark}\rm
Recall that the coefficients of the PDE (\ref{PDE}) are determined
by ${\cal L}_x$, and the \underline{restriction} of $\hat{m}$ to $(0,y_1)$.
Thus, the value of $m_0$ cannot be recovered  by taking right-hand
limits at zero of partial derivatives of solutions to (\ref{PDE}). 
This observation is important because the complete Bernstein function $\psi$
used in Corollary \ref{ourD2N} actually depends on $m_0$,
so that $-\psi(-{\cal L}_x)$ might not map 
the Dirichlet boundary condition to the Neumann boundary condition, 
if $m_0>0$, and we indeed see an $m_0$-related  correction-term 
in the conclusion of both 
Theorem \ref{trace_gen_theorem} and Corollary \ref{ourD2N}.
\end{remark}

By the same reasoning, (\ref{theirD2N}) should only be true
if the representation of the complete Bernstein function $\psi$
does not involve a positive $m_0$.
The minor mistake leading to this issue
can be traced back to \cite[Subsection A.3]{KM2018}:
in their notation, the authors claim that $-\phi'_\lambda(0)=\psi(\lambda)$,
which is not always true.
To correct their statement, a limit corresponding to 
our term $m_0\,{\cal L}_x u_f(x,0)$ in Corollary \ref{ourD2N} 
should be added on the right-hand side of (\ref{theirD2N}).

Otherwise, if $m_0=0$ and ${\cal L}_x=\Delta_x$,
our result under Remark \ref{DtoNtrue}(c) looks very similar
to (\ref{theirD2N}), though there are differences.
The obvious one is that the limit on the 
right-hand side of (\ref{theirD2N}) is in $L^2(\bfR^d)$,
while our limit 
$\partial_y^+ u_f(x,0)=\lim_{y\downarrow 0}\partial_y^+ u(x,y)$
is pointwise in $x\in\bfR^d$.
Furthermore, the solution $u$ used in (\ref{theirD2N}) solves
the PDE in $\bfR^d\times(0,R)$, while our solution solves
the PDE in $\bfR^d\times(0,y_1)$,
and our notion of solution is stronger.
Last but not least, equality (\ref{theirD2N}) holds true whenever
$f$ is in the $L^2(\bfR^d)$-domain of $-\psi(-\Delta_x)$,
while we require a whole range of assumptions for our result
of the same kind. However, we will see below that these 
are only two different pictures of the same result,
if the Dirichlet data $f=u_f(\cdot,0)$ is regular enough.

Let ${\cal S}(\bfR^d)$ be the Schwartz space of rapidly decreasing functions,
which is a core for $\Delta_x$ on $L^2(\bfR^d)$.
Fix a non-trivial Krein string $\hat{m}$,
choose $f\in{\cal S}(\bfR^d)$,
and define $\psi$ \&\ $\hat{u}$ via (\ref{ratioLim}) \&\ (\ref{KMsolution}), 
respectively.

We now verify that
\[
u(x,y)
\,=\,
\frac{1}{(2\pi)^d}\int_{\bfR^d}e^{{\rm i}\,x\cdot\xi}\,\hat{u}(\xi,y)\,\dd\xi,
\quad (x,y)\in\bfR^d\times(0,R),
\]
induces a solution to (\ref{PDE}) in the sense of Definition \ref{solution}
satisfying the conditions of Corollary \ref{ourD2N}.
For those properties of $\phi(\lambda,y)$ which we are going to use without
proof or further reference, the reader is referred to \cite[Subsection A.3]{KM2018}.

First, since $f\in{\cal S}(\bfR^d)$, and since
$\phi(\lambda,y)$ is bounded but also continuous in $y$,
all partial derivatives $\partial_{i,j}\,u,\,1\le i,j\le d$,
exist and are jointly continuous on $\bfR^d\times(0,R)$,
by dominated convergence. Therefore, the restriction of $u$ to
$\bfR^d\times(0,y_1)$ is a solution to 
(\ref{PDE}) in the sense of Definition \ref{solution},
and this solution is bounded.

Second, observe that 
$\lim_{y\downarrow 0}\phi(\lambda,y)=1$, for all $\lambda\ge 0$,
by definition. Furthermore, if $R<\infty$, then 
$\lim_{y\uparrow R}\phi(\lambda,y)=0$, for all $\lambda\ge 0$,
where $y_1=R$.
If $R=\infty$ but $y_1<R$,
then $\lim_{y\uparrow y_1}\phi(\lambda,y)$ exists,
because $\phi(\lambda,y)$ is continuous in $y$, for all $\lambda\ge 0$.
As a consequence, again by dominated convergence,
the solution $u$ satisfies Assumption \ref{BCxreg}.
Of course, all extended functions are even bounded and continuous on
$\bfR^d\times\left(\rule{0pt}{11pt}[0,y_1]\cap\bfR\right)$.
Note that $\psi(0)=1/R$, if $R<\infty$, and $\psi(0)=0$ otherwise.

Third, it follows from the definitions of $f_N$ \&\ $f_D$ that,
for any fixed $\lambda\ge 0$,
\[
\partial^+_y\phi(\lambda,y)
\,=\,
-\psi(\lambda)+\lambda\int_{[0,y]}\phi(\lambda,z)\,\hat{m}(\dd z),
\quad y\in(0,R),
\]
and hence
\[
|\partial^+_y\phi(\lambda,y)|
\,\le\,
\psi(\lambda)+\lambda\,\hat{m}(\,[0,1\wedge\frac{y_1}{2}]\,),
\quad\lambda\ge 0,\;y\in(0,1\wedge\frac{y_1}{2}),
\]
because $|\phi(\lambda,z)|$ is bounded by one.
Here, the complete Bernstein function $\psi(\lambda)$ grows linearly in $\lambda$,
which can easily be derived from its representation given in Lemma A.1.
Therefore, $\partial_y^+ u_f(x,0)\in C_b(\bfR^d)$ follows by dominated convergence.
Since $\mathcal{L}_x u_f(\cdot, 0)\in C_b(\bfR^d)$ has already been verified
when checking Assumption \ref{BCxreg}, the extension $u_f$ satisfies
$\partial^+_y u_f(\cdot, 0) 
+
m_0\,\mathcal{L}_x u_f(\cdot, 0)\in C_b(\mathbb{R}^d)$.

Fourth, in both (b1) and (b2) it holds that $y_1<\infty$ but $\hat{m}([0,y_1])=+\infty$.
Therefore, $y_1=R<\infty$ and
$\lim_{y\uparrow R}\phi(\lambda,y)=0$, for all $\lambda\ge 0$,
so that the respective Dirichlet boundary conditions at $y_1$
in the cases (b1) and (b2) of Theorem \ref{trace_gen_theorem}
are satisfied, by dominated convergence.

Fifth, if $y_1<\infty$ but $\hat{m}([0,y_1])<+\infty$,
then $R=\infty$ and $\hat{m}(y)=const<\infty$, for all $y\in[y_1,\infty)$.
Thus $\mathcal{L}_x u_f(\cdot, y_1-)\in C(\bfR^d)$
because all partial derivatives $\partial_{i,j}\,u,\,1\le i,j\le d$,
are jointly continuous on $\bfR^d\times(0,R)$.
By dominated convergence,
for $\partial_y^- u_f(\cdot,y_1-) \equiv m_1\,{\cal L}_x u_f(\cdot,y_1-)$,
it remains to show that
\begin{equation}\label{wanted}
\lim_{y\uparrow y_1}\partial^+_y\phi(\lambda, y)
\,=\,
-\lambda\,m_1\,\phi(\lambda, y_1),
\end{equation}
for any fixed $\lambda>0$.
Observe that the definitions of $f_N,f_D$ yield
\[
f_N(\lambda,y)
\,=\,
f_N(\lambda,y_1)
+
\lambda(y-y_1)\int_{[0,y_1]}f_N(\lambda,z)\,\hat{m}(\dd z),
\quad y\in[y_1,\infty),
\]
and
\[
f_D(\lambda,y)
\,=\,
f_D(\lambda,y_1)
+
(y-y_1)
+
\lambda(y-y_1)\int_{[0,y_1]}f_D(\lambda,z)\,\hat{m}(\dd z),
\quad y\in[y_1,\infty),
\]
so that
\[
\psi(\lambda)
\,=\,
\frac{
\lambda\int_{[0,y_1]}f_N(\lambda,z)\,\hat{m}(\dd z)
}
{
1+\lambda\int_{[0,y_1]}f_D(\lambda,z)\,\hat{m}(\dd z)
}\,.
\]
On the other hand, since
\[
\lim_{y\uparrow y_1}\partial^+_y\phi(\lambda, y)
\,=\,
-\psi(\lambda)+\lambda\int_{[0,y_1)}\phi(\lambda,z)\,\hat{m}(\dd z),
\]
the wanted equality (\ref{wanted}) is equivalent to
\[
\psi(\lambda)
\,=\,
\lambda\int_{[0,y_1]}\phi(\lambda,z)\,\hat{m}(\dd z)\\
\,=\,
\lambda\int_{[0,y_1]}f_N(\lambda,z)\,\hat{m}(\dd z)
-
\lambda\psi(\lambda)\int_{[0,y_1]}f_D(\lambda,z)\,\hat{m}(\dd z),
\]
which is indeed true for $\psi(\lambda)$ as obtained above.
All in all, the boundary condition at $y_1$
required in case (c) of Theorem \ref{trace_gen_theorem}
is also satisfied.
\begin{remark}\rm
The above five steps, which are valid for arbitrary Krein strings $\hat{m}$,
show that, if the Dirichlet data $f$ is regular enough, 
then the solution $u$ constructed in \cite{KM2018} satisfies both
the condition for (\ref{theirD2N})
and all our conditions for 
Theorem \ref{trace_gen_theorem} \&\ Corollary \ref{ourD2N}.
For the purpose of demonstration,
we have chosen Dirichlet data $f$ from ${\cal S}(\bfR^d)$
which is a popular core for the generator $\Delta_x$ 
in several Banach spaces.
\end{remark}

Concluding this section, we eventually compare our results
with the results in \cite{K2019},
where Kwa\'snicki generalises his work with Mucha, \cite{KM2018},
to elliptic PDEs of type
\[
\partial_x^2 u\times\hat{m}(\dd y)+2\partial_{xy}u\times b(y)+\partial_y^2 u
\,=\,
0, \quad\mbox{in}\quad\bfR\times(0,R).
\]
The main difference to our work is that the coefficients of such PDEs
only depend on $y$. However, it would be very interesting to
study the Dirichlet-to-Neumann map associated with this type of PDE
via trace processes.
\section*{Appendix}
Recall the Krein string $\hat{m}$,
as well as the speed measure $m$ of the diffusion $Y$, introduced in Section \ref{rr}.
Again, $\hat{m}$ is used to denote both the string 
and the corresponding measure on $([0,y_1]\cap\bfR,{\cal B}([0,y_1]\cap\bfR))$.

Any Krein string can be uniquely associated with
a so-called characteristic function, by Krein's correspondence, \cite{K1952},
and this correspondence can be used to understand the structure 
of certain Laplace exponents.\\

\noindent
{\bf Lemma A.1.}\,{\it
The right-inverse local time $[L^0_t(Y)^{-1},\,t\ge 0]$ 
is a subordinator with Laplace exponent
\[
\psi(\lambda)\,=\,
\mu(\{0\})+m_0\,\lambda\,+\int_{(0,\infty)}\frac{\lambda}{\lambda+\lambda'}\,\mu(\dd\lambda'),
\quad\lambda>0,
\]
where $\mu$ is the measure used in the representation of
Krein's characteristic function associated with the right-inverse string $\hat{m}^{-1}$.
This measure satisfies $\int_{[0,\infty)}\mu(\dd\lambda')/(1+\lambda')<+\infty$.
}
\begin{proof}
This result was proven in \cite{KW1982}, by Kotani \&\ Watanabe; 
we only have to clarify how the notation used in their proof relates to our notation
taken from \cite{AS1998}.
For this purpose, 
we extend our measure $\hat{m}$ to a measure on $(\bfR,{\cal B}(\bfR))$ by
\[
\hat{m}(O)\,\stackrel{def}{=}\,0,
\quad\mbox{for any open subset $O\subseteq(-\infty,0)\cup(y_1,\infty)$}.
\]

Observe that,
in both \cite[Chapter VI]{AS1998} and \cite{KW1982}, 
our diffusion $Y$ would be constructed 
by random time change of a Wiener process, $[W(t),\,t\ge 0]$, 
and the definition of the random time change would be based on the local time
of this Wiener process. 
Indeed, in both \cite[Chapter VI]{AS1998} and \cite{KW1982}, 
the authors first define
\[
A_t\,=\,\int_\bfR\frac{1}{2}\,L_t^y(W)\,\hat{m}(\dd y),\quad t\ge 0,
\]
and then construct $Y(t)=W(A_t^{-1}),\,t\ge 0$,
where $[A_t^{-1},\,t\ge 0]$ is the right-inverse of $[A_t,\,t\ge 0]$.
Here, we used the local time from \cite{AS1998}.
Note that the local time used in \cite{KW1982}
is half of our local time, and that the measure $m$ they used is NOT
the speed measure as in the present paper but 
the extension of $\hat{m}$ to a measure on $(\bfR,{\cal B}(\bfR))$, as defined above.

Furthermore, \cite[Theorem (5.27)]{AS1998} implies that
\[
\int_0^t{\bf 1}_\Gamma(Y(s))\,\dd s
\,=\,
\int_\Gamma L_t^y(Y)\,m(\dd y),
\quad t\ge 0,\;\mbox{a.s.},
\]
for any Borel set $\Gamma\subseteq[0,y_1]\cap\bfR$,
and this equality specifies the local time of $Y$ used in \cite{AS1998},
based on the speed measure $m$.
But, if $l_t^y(Y)$ denotes the local time of $Y$ used in \cite{KW1982}, then
\[
\int_0^t{\bf 1}_\Gamma(Y(s))\,\dd s
\,=\,
\int_\Gamma l_t^y(Y)\,\hat{m}(\dd y),
\quad t\ge 0,\;\mbox{a.s.},
\]
for any Borel set $\Gamma\subseteq[0,y_1]\cap\bfR$.

Therefore, since $m(\{0\})=\hat{m}(\{0\})$, if $m_0>0$,
then $L_t^0(Y)=l_t^0(Y),\,t\ge 0$, a.s., is obvious.
However, being an easy consequence of \cite[Lemma (6.34)]{AS1998},
this equality of the two local times at \underline{zero} also holds true when $m_0=0$.
Note that, $L_t^y(Y)=2\,l_t^y(Y),\,t\ge 0$, a.s., for all $y\in(0,y_1)$.
The case $y=0$ is different because zero is a special singular point for $Y$.
If $y_1$ was finite, then it would be a special singular point, too. 

All in all, the Laplace exponent obtained for the right-inverse of $l_t^0(Y)$ in \cite{KW1982} 
is indeed identical to the Laplace exponent of our right-inverse local time $L_t^0(Y)^{-1}$.
\end{proof}

\begin{proof}[Proof of \textbf{Corollary \ref{BCyreg}}]
We first show that $\partial_y^+ u(\cdot,y^\star)$ is locally bounded on $\bfR^d$,
for an arbitrary but fixed $y^\star\in(0,y_1)$.

Fix $y^\star\in(0,y_1)$, and assume the contrary. Then there exists a sequence
$(x_n)_{n=1}^\infty\subseteq\bfR^d$, which converges to some $x\in\bfR^d$,
such that $\sup_{n}|\partial_y^+ u(x_n,y^\star)|=+\infty$.
Without loss of generality, assume that
\[
\forall\,r>0\;\exists\,n_r\;\forall\,n\ge n_r:\;
\partial_y^+ u(x_n,y^\star)\,\ge\,r.
\]

Next, fix $y_\star\in(0,y^\star)$, and note that
\[
\partial_y^+ u(x_n,y)
\,=\,
\partial_y^+ u(x_n,y^\star)
\,+
\int_{(y,y^\star]}{\cal L}_x u(x_n,y')\,\hat{m}(\dd y'),
\]
for all $n\ge 0$, and all $y\in[y_\star,y^\star)$,
is an easy consequence of (\ref{L dm integral}).
By Assumption \ref{BCxreg}, but also using $\hat{m}(K)<+\infty$,
for any compact subset $K\subseteq[0,y_1)$,
as well as continuity of the coefficients of ${\cal L}_x$,
\[
c(y_\star,y^\star)
\,\stackrel{\mbox{\tiny def}}{=}\,
\sup_{n\ge 0}\sup_{y\in[y_\star,y^\star)}
|\int_{(y,y^\star]}{\cal L}_x u(x_n,y')\,\hat{m}(\dd y')\,|\,<\,\infty.
\]

Then, for any $r>0$,
\[
\partial_y^+ u(x_n,y)\,\ge\,r-c(y_\star,y^\star),
\quad
\forall\,n\ge n_r,\;\forall\,y\in[y_\star,y^\star],
\]
and hence
\begin{align*}
u(x_n,y_\star)
&\,=\,
u(x_n,y^\star)\,-\int_{y_\star}^{y^\star}\partial_y^+ u(x_n,y)\,\dd y\\
&\,\le\,
u(x_n,y^\star)+[c(y_\star,y^\star)-r]\times(y^\star-y_\star),
\quad\forall\,n\ge n_r.
\end{align*}
Of course, $\sup_{n\ge 0}|u(x_n,y^\star)|<+\infty$
since $u(\cdot,y^\star)$ is continuous and $x_n\to x\in\bfR^d,\,n\to\infty$,
so that $\limsup_{n\to\infty}u(x_n,y_\star)=-\infty$,
which contradicts the continuity of $u(\cdot,y_\star)$.

All in all $\partial_y^+ u(\cdot,y^\star)$ is indeed locally bounded on $\bfR^d$.
Therefore,
\[
\partial_y^+ u_f(x,0)
\,=\,
\lim_{y\downarrow 0}\,\partial_y^+ u(x,y)
\,=\,
\mbox{$\lim_{y\downarrow 0}$}
\left(
\partial_y^+ u(x,y^\star)
\,+
\int_{(y,y^\star]}{\cal L}_x u(x,y')\,\hat{m}(\dd y')
\right),\quad x\in\bfR^d,
\]
is a locally bounded function on $\bfR^d$, too,
since ${\cal L}_x u_f$ is locally bounded on $\bfR^d\times[0,y_1)$.

Obviously, for fixed $x\in\bfR^d$, the extended version $\partial_y^+ u_f(x,\cdot)$
defined this way is right-continuous at zero, and thus it is c\`adl\`ag on $[0,y_1)$
because $u(x,\cdot)$ is a difference of two convex functions on the interior $(0,y_1)$,
finishing the proof of the corollary in the case of $\partial_y^+ u_f$.

In the case of $\partial_y^- u_f$, the extension can be given by
\[
\partial_y^- u_f(x,y_1-)
\,=\,
\mbox{$\lim_{y\uparrow y_1}$}
\left(
\partial_y^+ u(x,y^\star)
\,-
\int_{(y^\star,y]}{\cal L}_x u(x,y')\,\hat{m}(\dd y')
\right),\quad x\in\bfR^d,
\]
though we omit the proof. 
\end{proof}

\end{document}